\documentclass[12pt]{article}
\usepackage{amscd}
\usepackage[all]{xy}
\usepackage{CJK}
\usepackage{CJKnumb}
\usepackage{graphicx}
\usepackage{amsfonts}
\usepackage{amsmath}
\usepackage{amssymb}
\usepackage{fancyhdr}
\usepackage{indentfirst}
\usepackage{titlesec}
\usepackage{mathrsfs}
\usepackage{bbm}
\usepackage{dsfont}
\usepackage{amsthm} 
\usepackage{mathrsfs}

\newcommand{\cP}{\mathcal{P}}
\newcommand{\cM}{\mathcal{M}}

\newcommand{\cQ}{\mathcal{Q}}
\newcommand{\cR}{\mathcal{R}}

\newcommand{\cL}{\mathcal{L}}
\newcommand{\cN}{\mathcal{N}}
\newcommand{\cA}{\mathcal{A}}

\newcommand{\cC}{\mathcal{C}}

\newcommand{\cB}{\mathcal{B}}
\newcommand{\cK}{\mathcal{K}}
\newcommand{\cH}{\mathcal{H}}

\theoremstyle{definition}  \newtheorem{Def}{Definition}
\theoremstyle{plain}  \newtheorem{Thm}{Theorem} \theoremstyle{plain}
\newtheorem{cor}{Corollary}
 \theoremstyle{remark} \newtheorem{Rek}{Remark}
 \theoremstyle{plain}\newtheorem{Lem}{Lemma}
 \theoremstyle{plain}
\newtheorem{Prop}{Proposition}
\theoremstyle{remark}

\begin{document}

\title{\textbf{Higher Dimensional Homology Algebra III:Projective Resolutions and Derived 2-Functors in (2-SGp)}}
\author{Fang HUANG, Shao-Han CHEN, Wei CHEN, Zhu-Jun ZHENG\thanks{Supported in part by NSFC with grant Number
10971071
 and Provincial Foundation of Innovative
Scholars of Henan.} }
\date{}
 \maketitle

\begin{center}
\begin{minipage}{5in}
{\bf  Abstract}: In this paper, we will define the derived 2-functor
by projective resolution of any symmetric 2-group, and give some
related properties of the derived 2-functor.

\\
\end{minipage}
\end{center}
\maketitle \hspace{1cm}

\section{Introduction}

In recent years, higher dimensional category theory has been largely
developed from a series of analogies with the potential
applications. For instance, in the representation theory, the
representation spaces not only to be vector spaces, but also to be
categories (or even higher categories)(\cite{30}), such as the
representation of categorical group, algebraic group(\cite{22,30}),
using category representations to describe the topological quantum
field theory(\cite{23}) and so on. In algebraic geometry, J. Lurie
gives a very tractable model of $(\infty,1)$-categories
(\cite{8,9,25}), and also A. Joyal's important work \cite{26}
showing that one can do category theory in quasi-categories is an
essential precursor to Lurie¡¯s work and is unquestionably one of
the most important recent developments in higher category theory.
Lie 2-algebra gives a solution of the Zamolodchikov tetrahedron
equation \cite{31} and Lie 2-group admits self-dual solutions in
five-dimensional space time in higher Yang-Mills theory and 2-form
electromagnetism\cite{32}. L.Breen's paper\cite{21} gives an idea of
how naturally 2-categorical algebra arises in the study of algebraic
geometry and differential geometry, such as Lie algebroids\cite{29},
integration of Lie 2-algebra\cite{28}, which are interesting
researching subjects. Higher dimensional category theory also has
been applied in algebraic topology theory\cite{9}, computer science,
logic etc..


In \cite{11}, A.del R\'{\i}o, J. Mart\'{\i}nez-Moreno and E. M.
Vitale gave the definition of cohomology categorical groups for any
complex in the 2-category (2-SGp)(which is an abelian
2-category\cite{2}) of symmetric categorical groups(we call them
symmetric 2-groups) after discussing the relative kernel and
relative cokernel, and constructed a long 2-exact sequence from an
extension of complexes in (2-SGp). These drive us to write a series
of papers to develop a homological algebra for 2-categories (2-SGp)
and ($\cR$-2-Mod)(\cite{4}).

This is the third paper of the series. In our first paper\cite{4} of
this series, we gave the definition of $\cR$-2-module. In the second
paper\cite{14}, we proved that the 2-categories (2-SGp) and
($\cR$-2-Mod) are projective enough. In this paper, we shall give
the definition of left derived 2-functor for the 2-category (2-SGp)
and give a fundamental property of derived 2-functor. When we
finished this paper, we found Prof. T.Pirashvili also discussed some
problems about higher homological theory\cite{20,24}.

For a symmetric 2-group, we construct a projective resolution in the
2-category (2-SGp) and prove that it is unique up to 2-chain
homotopy (Proposition 2 and Theorem 1). These results are the main
stones of this paper and make it possible to define left derived
2-functor in (2-SGp).

In 1-dimensional case, derived functor has many applications in many
fields of mathematics, such as ring theory, algebraic topology,
representation theory, algebraic geometry
etc.\cite{7,10,12,16,17,18}. We believe that derived 2-functor
should have many applications in higher dimensional category theory.



The present paper is organized as follows. In section 2, we recall
some definitions in (2-SGp) such as the relative (co)kernel,
relative 2-exact which are appeared in \cite{2,11,6}. By the similar
method in \cite{11}, we give the definitions of homology symmetric
2-groups for a complex of symmetric 2-groups and describe them
explicitly, show the induced morphisms of homology symmetric
2-groups more clearly. We also give the definition of 2-chain
homotopy of two morphisms of complexes in (2-SGp) like chain
homotopy in 1-dimensional case, and prove that it induces an
equivalent morphisms between homology symmetric 2-groups. In section
3, we mainly give the definition of projective resolution of a
symmetric 2-group and its construction(Proposition 2). In the last
section, we define the left derived 2-functor and obtain our main
result Theorem 2.

\section{Preliminary}
In this section, we review the constructions of the relative
(co)kernel and the definition of relative 2-exactness of a
sequence\cite{2,11}, and then give the homology symmetric 2-groups
of a complex of symmetric 2-groups similar to the cohomology 2-group
given in \cite{11}.
The relative kernel\cite{11}
$(Ker(F,\varphi),e_{(F,\varphi)},\varepsilon_{(F,\varphi)})$ of a
sequence $(F,\varphi,G):\cA\rightarrow\cB\rightarrow\cC$ in (2-SGp)
is a symmetric 2-group consisting of:

$\cdot$ An object is a pair $(A\in obj(\cA),a:F(A)\rightarrow 0)$
such that the following diagram commutes
\begin{center}
\scalebox{0.9}[0.85]{\includegraphics{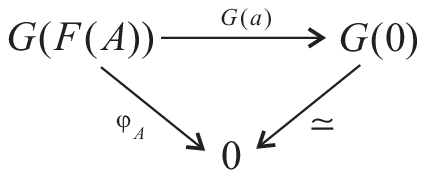}}
\end{center}
$\cdot$ A morphism $f:(A,a)\rightarrow (A^{'},a^{'})$ is a morphism
$f:A\rightarrow A^{'}$ in $\cA$ such that the following diagram
commutes
\begin{center}
\scalebox{0.9}[0.85]{\includegraphics{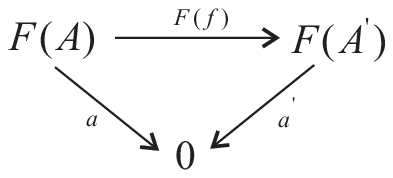}}
\end{center}
$\cdot$ The faithful functor
$e_{(F,\varphi)}:Ker(F,\varphi)\rightarrow\cA$ is defined by
$e_{(F,\varphi)}(A,a)=A$, and the natural transformation
$\varepsilon_{(F,\varphi)}:F\circ e_{(F,\varphi)}\Rightarrow 0$ by
$(\varepsilon_{(F,\varphi)})_{(A,a)}=a$.

The relative cokernel\cite{11}
$(Coker(\varphi,G),p_{(\varphi,G)},\pi_{(\varphi,G)})$ of a sequence
$(F,\varphi,G):\cA\rightarrow\cB\rightarrow\cC$ in (2-SGp) is a
symmetric 2-group consisting of:

$\cdot$  Objects are those of $\cC$.

$\cdot$ A morphism from $X$ to $Y$ is an equivalent class of pair
$(B,f):X\rightarrow Y$ with $B\in obj(\cB)$ and $f:X\rightarrow
G(B)+Y$. Two morphisms $(B,f),(B^{'},f^{'}):X\rightarrow Y$ are
equivalent if there is $A\in obj(\cA)$ and $a:B\rightarrow
F(A)+B^{'}$ such that the following diagram commutes
\begin{center}
\scalebox{0.9}[0.85]{\includegraphics{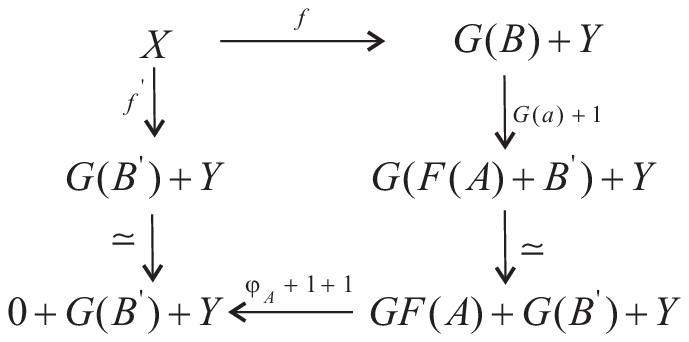}}
\end{center}
$\cdot$ The essentially surjective functor
$p_{(\varphi,G)}:\cC\rightarrow Coker(\varphi,G)$ is defined by
$p_{(\varphi,G)}(X)=X$, and the natural transformation
$\pi_{(\varphi,G)}: p_{(\varphi,G)}\circ G\Rightarrow 0$ by
$(\pi_{(\varphi,G)})_{B}=1_{G(B)}$.

The universal properties of relative kernel and cokernel just like
the usual ones, more details see \cite{11}.


\begin{Def}(\cite{11})
Consider the following diagram in (2-SGp)
\begin{center}
\scalebox{0.9}[0.85]{\includegraphics{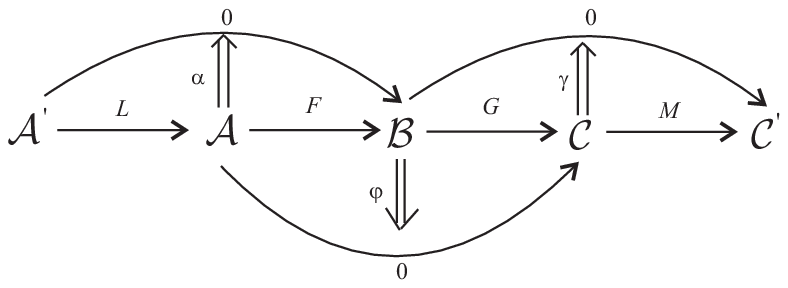}}
\end{center}
with $\alpha$ compatible with $\varphi$ and $\varphi$ compatible
with $\gamma$. By the universal property of the relative kernel
$Ker(G,\gamma)$, we get a factorization $(F^{'},\varphi^{'})$ of
$(F,\varphi)$ through $(e_{(F,\varphi)},\varepsilon_{(F,\varphi)})$.
By the cancellation property of $e_{(F,\varphi)}$, we have a
2-morphism $\overline{\alpha}$ as in the following diagram
\begin{center}
\scalebox{0.9}[0.85]{\includegraphics{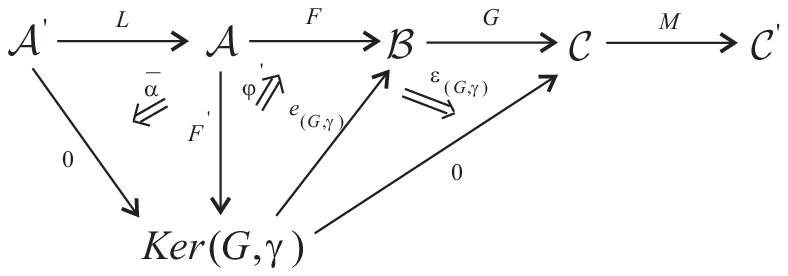}}
\end{center}
We say that the sequence $(L,\alpha,F,\varphi,G,\gamma,M)$ is
relative 2-exact in $\cB$ if the functor $F^{'}$ is essentially
surjective and $\overline{\alpha}$-full.
\end{Def}
\begin{Rek}
The equivalent definition of relative 2-exact is also given in
\cite{11}.
\end{Rek}

In the following, we will omit the composition symbol $\circ$ in our
diagrams.

From \cite{3,13}, a complex of symmetric 2-groups is a diagram in
(2-SGp) of the form
$$
\cA_{\cdot}=\cdot\cdot\cdot\xrightarrow[]{L_{n+1}}\cA_{n}\xrightarrow[]
{L_{n}}\cA_{n-1}\xrightarrow[]{L_{n-1}}\cA_{n-2}\xrightarrow[]{L_{n-2}}\cdot\cdot\cdot
\xrightarrow[]{L_{2}}\cA_{1}\xrightarrow[]{L_{1}}\cA_{0}
$$
together with a family of 2-morphisms $\{\alpha_{n}:L_{n-1}\circ
L_{n}\Rightarrow 0\}_{n\geq2}$ such that, for all $n$, the following
diagram commutes
\begin{center}
\scalebox{0.9}[0.85]{\includegraphics{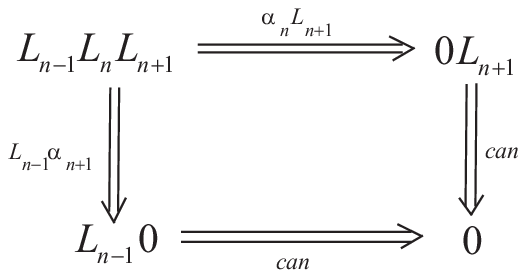}}
\end{center}
We call it 2-chain complex in our papers.

Consider part of the complex
\begin{center}
\scalebox{0.9}[0.85]{\includegraphics{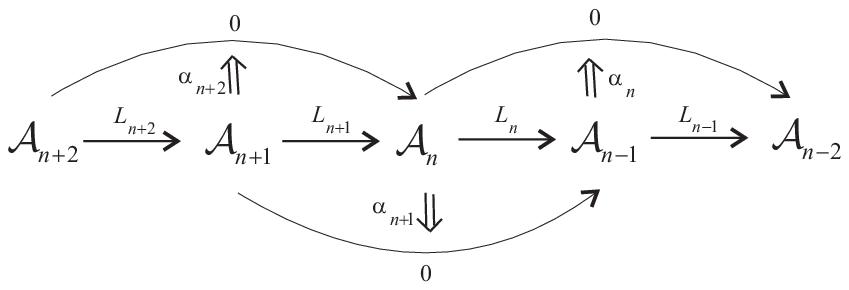}}
\end{center}
Based on the properties of relative kernel $Ker(L_n,\alpha_n)$, we
have the following diagram
\begin{center}
\scalebox{0.9}[0.85]{\includegraphics{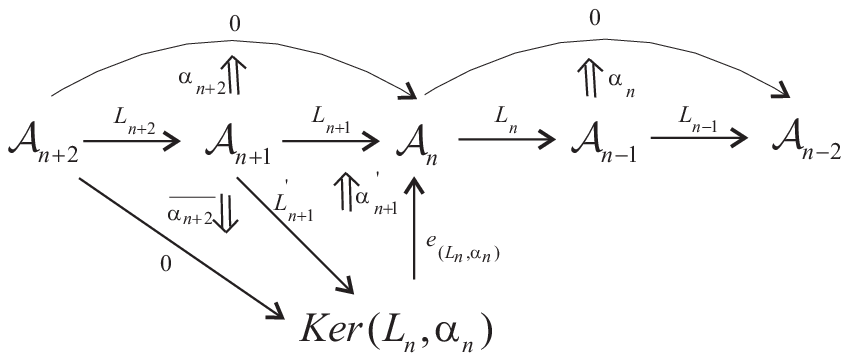}}
\end{center}
Similar to the definition of cohomology 2-group in \cite{11}, the
$n$th homology symmetric 2-group $\cH_{n}(\cA_{\cdot})$ of the
complex $\cA_{\cdot}$ is defined as the relative cokernel
$Coker(\overline{\alpha_{n+2}},L_{n+1}^{'})$.

Note that, to get $\cH_{0}(\cA_{\cdot})$ and $\cH_{1}(\cA_{\cdot})$,
we have to complete the complex $\cA_{\cdot}$ on the right with the
two zero-morphisms and two canonical 2-morphisms\\
$
\cdot\cdot\cdot\xrightarrow{L_2}\cA_1\xrightarrow[]{L_1}\cA_{0}\xrightarrow[]{0}0\xrightarrow[]{0}0,
$ can : $0\circ L_1\Rightarrow 0,$ can : $0\circ 0\Rightarrow 0.$

We give an explicit description of $\cH_{n}(\cA_{\cdot})$ following
from the cohomology symmetric 2-group given in \cite{11}.

$\cdot$ an object of $\cH_{n}(\cA_{\cdot})$ is an object of the
relative kernel $Ker(L_{n},\alpha_n)$, that is a pair
$$
(A_n\in obj(\cA_n), a_n:L_n(A_n)\rightarrow 0)
$$
such that $L_{n-1}(a_n)=(\alpha_{n})_{A_n}$;

$\cdot$ a morphism $(A_n,a_n)\rightarrow (A_n^{'},a_n^{'})$ is an
equivalent pair
$$
(X_{n+1}\in obj(\cA_{n+1}),x_{n+1}:A_n\rightarrow
L_{n+1}(X_{n+1})+A_n^{'})
$$
such that the following diagram commutes
\begin{center}
\scalebox{0.9}[0.85]{\includegraphics{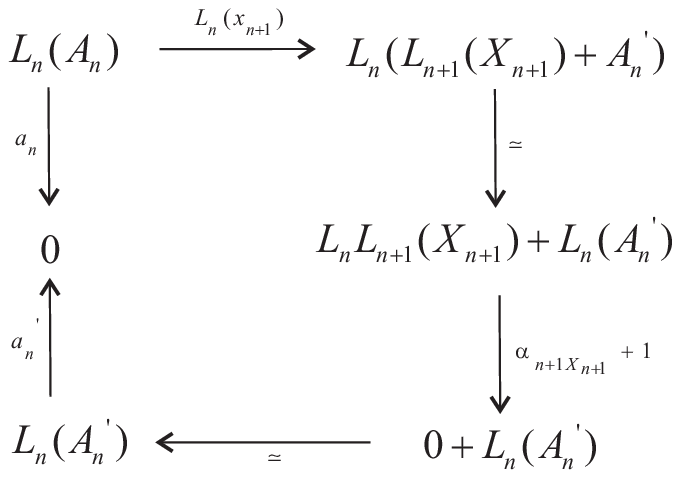}}
\end{center}
Two morphisms
$(X_{n+1},x_{n+1}),(X_{n+1}^{'},x_{n+1}^{'}):(A_n,a_n)\rightarrow
(A_n^{'},a_n^{'})$ are equivalent if there is a pair
$$
(X_{n+2}\in obj(\cA_{n+2}),x_{n+2}:X_{n+1}\rightarrow
L_{n+2}(X_{n+2})+X_{n+1}^{'})
$$
such that the following diagram commutes
\begin{center}
\scalebox{0.9}[0.85]{\includegraphics{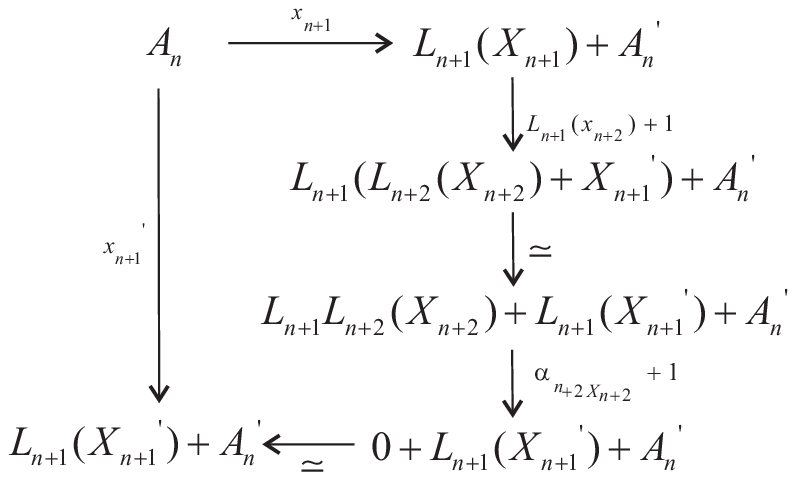}}
\end{center}

Similar to \cite{11}, a morphism
$(F_{\cdot},\lambda_{\cdot}):\cA_{\cdot}\rightarrow\cB_{\cdot}$ of
complexes in (2-SGp) is a picture in the following diagram
\begin{center}
\scalebox{0.9}[0.85]{\includegraphics{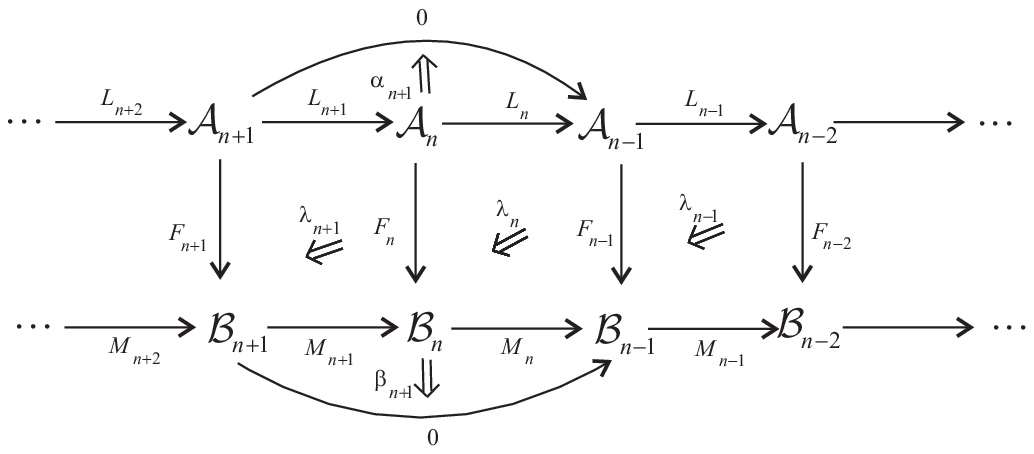}}
\end{center}
where $F_{n}:\cA_n\rightarrow \cB_n$ is 1-morphism in (2-SGp),
$\lambda_{n}:F_{n-1}\circ L_{n}\Rightarrow M_{n}\circ F_n$ is
2-morphism in (2-SGp), for each $n$, making the following diagram
commutative
\begin{center}
\scalebox{0.9}[0.85]{\includegraphics{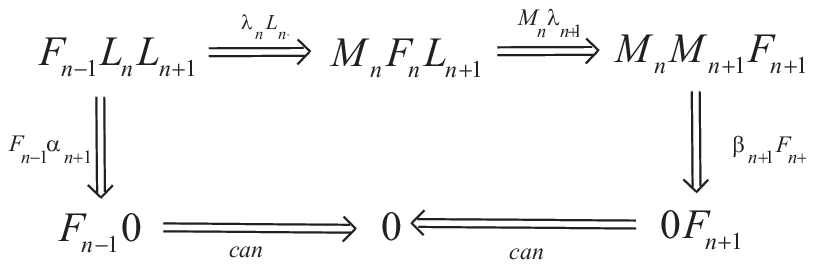}}
\end{center}
Such a morphism induces, for each $n$, a morphism of homology
symmetric 2-groups
$\cH_{n}(F_{\cdot}):\cH_{n}(\cA_{\cdot})\rightarrow
\cH_{n}(\cB_{\cdot})$ from the universal properties of relative
kernel and cokernel. It can be described explicitly.

Given an object $(A_n\in obj(\cA_{n}),a_n:L_n(A_n)\rightarrow 0)$ of
$\cH_{n}(\cA_{\cdot})$ with $L_{n-1}(a_n)=(\alpha_n)_{A_n}$, we have
$\cH_{n}(F_{\cdot})(A_n,a_n)=(F_n(A_n)\in
obj(\cB_{n}),b_n:M_{n}(F_n(A_n))\rightarrow 0)$, where $b_n$ is the
composition
$M_{n}(F_n(A_n))\xrightarrow[]{(\lambda_n)_{A_n}^{-1}}F_{n-1}L_{n}(A_n)\\
\xrightarrow[]{F_{n-1}(a_n)}F_{n-1}(0)\backsimeq 0$, together with
$M_{n-1}(b_n)=(\beta_n)_{F_n(A_n)}$. In fact, from the commutative
diagram of $\lambda_n$, we have the following commutative diagram
\begin{center}
\scalebox{0.9}[0.85]{\includegraphics{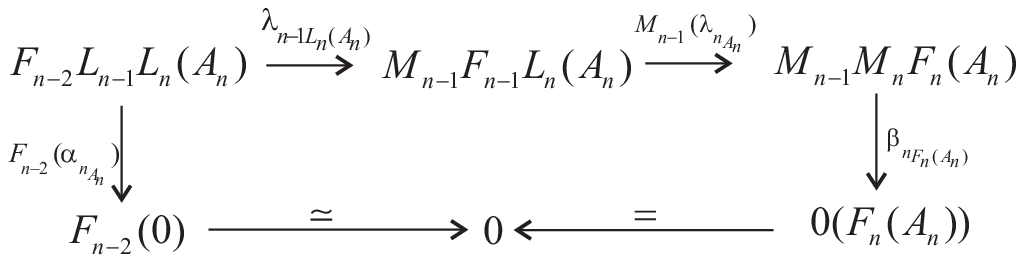}}
\end{center}
Moreover, consider 2-morphism $\lambda_{n-1}:F_{n-2}\circ
L_{n-1}\Rightarrow M_{n-1}\circ F_{n-1}$ and a morphism
$a_n:L_n(A_n)\rightarrow 0$ in $\cA_{n-1}$, we have the following
commutative diagram
\begin{center}
\scalebox{0.9}[0.85]{\includegraphics{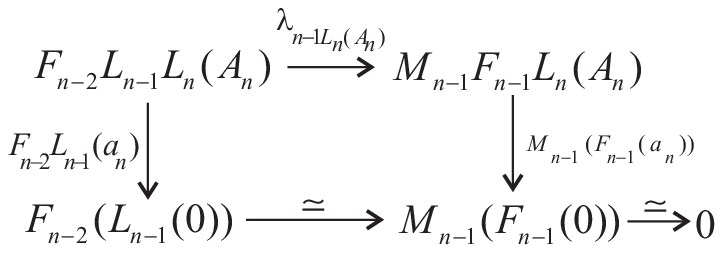}}
\end{center}
Then, by the above two diagrams, we have
$M_{n-1}(b_n)=(\beta_n)_{F_n(A_n)}$, i.e. $(F_n(A_n),b_n)$ is an
object of $\cH_{n}(\cB_{\cdot})$.

Given a morphism $[X_{n+1}\in obj(\cA_{n+1}),x_{n+1}:A_n\rightarrow
L_{n-1}(X_{n+1})+A_{n}^{'}]:(A_n,a_n)\rightarrow(A_{n}^{'},a_{n}^{'})$
in $\cH_{n}(\cA_{\cdot})$, satisfying the condition as in the above
definition. We have $
\cH_{n}(F_{\cdot})[X_{n+1},x_{n+1}]=[F_{n+1}(X_{n+1})\in
obj(\cB_{n+1}),\overline{x_{n+1}}:F_{n}(A_n)\rightarrow
M_{n+1}(F_{n+1}(X_{n+1}))+F_{n}(A_{n}^{'})]:(F_{n}(A_{n},b_n)\rightarrow(F_n(A_{n}^{'}),b_{n}^{'}),
$ where $\overline{x_{n+1}}$ is the composition
$F_{n}(A_n)\xrightarrow[]{F_{n}(x_{n+1})}F_{n}(L_{n+1}(X_{n+1}+A_{n}^{'}))\backsimeq
F_{n}L_{n+1}(X_{n+1})+F_n(A_{n}^{'})\xrightarrow[]{(\lambda_{n+1})_{X_{n+1}}+1}M_{n+1}(F_{n+1}(X_{n+1}))+F_{n}(A_{n}^{'})
$ and such that the following diagram commutes
\begin{center}
\scalebox{0.9}[0.85]{\includegraphics{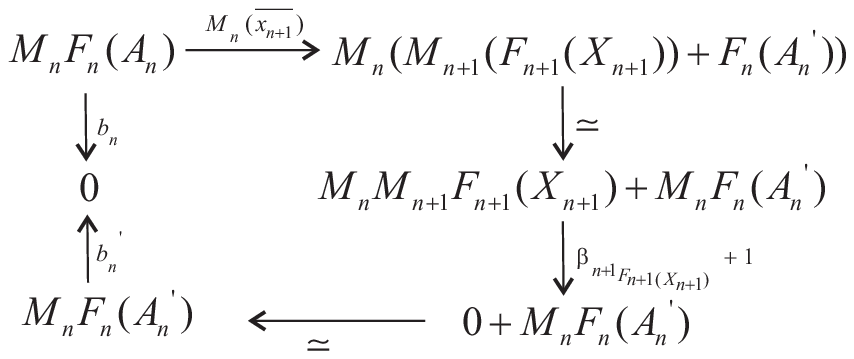}}
\end{center}
In fact, we have the following commutative diagrams
\begin{center}
\scalebox{0.9}[0.85]{\includegraphics{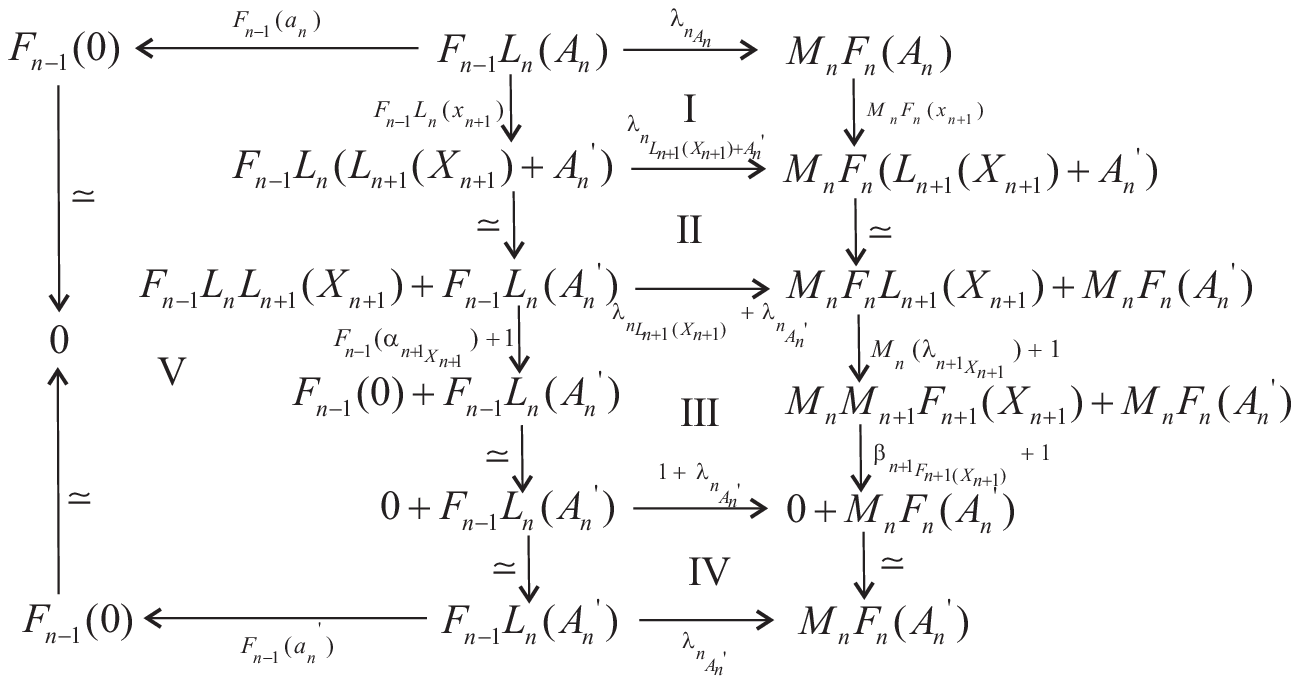}}
\end{center}
The commutativity of I follows from $\lambda_{n}$ is a natural
transformation. The commutativity of II follows from $\lambda_{n}$
is a 2-morphism. The commutativity of III follows from the
commutativity of $\lambda_{n}$ in definition. The commutativity of
IV is obvious. The commutativity of V follows from the operation of
$F_{n-1}$ on the commutative diagram of $[X_{n+1},x_{n+1}]$.

$\cH_{n}(F_{\cdot})$ is a morphism in (2-SGp) follows from the
properties of $F_{n}$.



\begin{Rek}
1. For a complex of symmetric 2-groups which is relative 2-exact in
each point, the (co)homology symmetric 2-groups are always zero
symmetric 2-group(only one object and one morphism)(\cite{11}).

2. For morphisms
$\cA_{\cdot}\xrightarrow[]{(F_{\cdot},\lambda_{\cdot})}\cB_{\cdot}\xrightarrow[]{(G_{\cdot},\mu_{\cdot})}\cC_{\cdot}$
of complexes of symmetric 2-groups, their composite is given by
$(G_{n}\circ F_{n},(\mu_n\circ F_{n+1})\star (G_n\circ \lambda_n))$,
for $n\in \mathds{Z}$, where $\star$ is the vertical composition of
2-morphisms in 2-category(\cite{19}). Moreover,
$\cH_{n}(G_{\cdot}\circ F_{\cdot})\backsimeq \cH_{n}(G_{\cdot})\circ
\cH_n(F_{\cdot})$ of homology symmetric 2-groups.
\end{Rek}

\begin{Def}
Let $(F_{\cdot},\lambda_{\cdot}),
(G_{\cdot},\mu_{\cdot}):(\cA_{\cdot},L_{\cdot},\alpha_{\cdot})\rightarrow
(\cB_{\cdot},M_{\cdot},\beta_{\cdot})$ be two morphisms of 2-chain
complexes of symmetric 2-groups. If there is a family of 1-morphisms
$\{H_{n}:\cA_{n}\rightarrow\cB_{n+1}\}_{n\in \mathds{Z}}$ and a
family of 2-morphisms $\{\tau_{n}:F_{n}\Rightarrow M_{n+1}\circ
H_{n}+H_{n-1}\circ L_n+G_{n}:\cA_{n}\rightarrow \cB_n\}_{n\in
\mathds{Z}}$ satisfying the obvious compatible conditions, i.e. the
following diagram commutes
\begin{center}
\scalebox{0.9}[0.85]{\includegraphics{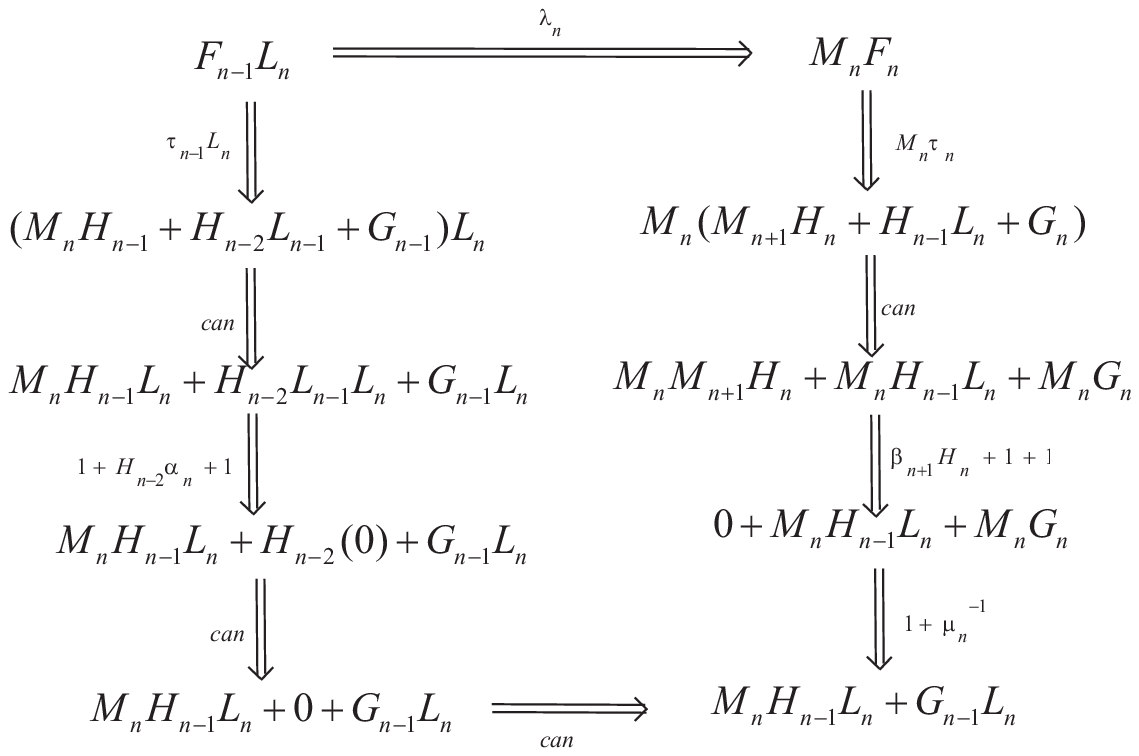}}
\end{center}
We call the above morphisms $(F_{\cdot},\lambda_{\cdot}),
(G_{\cdot},\mu_{\cdot})$ are 2-chain homotopy.
\end{Def}

\begin{Prop}
Let $(F_{\cdot},\lambda_{\cdot}),
(G_{\cdot},\mu_{\cdot}):(\cA_{\cdot},L_{\cdot},\alpha_{\cdot})\rightarrow
(\cB_{\cdot},M_{\cdot},\beta_{\cdot})$ be two morphisms of 2-chain
complexes of symmetric 2-groups. If they are 2-chain homotopy, there
is an equivalence $\cH_{n}(F_{\cdot})\backsimeq \cH_{n}(G_{\cdot})$
between induced morphisms.
\end{Prop}
\begin{proof}
In order to prove the equivalence between two morphisms, it will
suffice to construct a 2-morphism
$\varphi_n:\cH_{n}(F_{\cdot})\Rightarrow \cH_{n}(G_{\cdot})$, for
each $n$.

There are induced morphisms
\begin{align*}
&\cH_{n}(F_{\cdot}):\cH_{n}(\cA_{\cdot})\rightarrow \cH_{n}(\cB_{\cdot})\\
&\hspace{1.5cm}(A_n,a_n)\mapsto(F_n (A_n),b_n),\\
&\hspace{0.8cm}[X_{n+1},x_{n+1}]\mapsto
[F_{n+1}(X_{n+1}),\overline{x_{n+1}}]
\end{align*}
and
\begin{align*}
&\cH_{n}(G_{\cdot}):\cH_{n}(\cA_{\cdot})\rightarrow \cH_{n}(\cB_{\cdot})\\
&\hspace{1.5cm}(A_n,a_n)\mapsto(G_n (A_n),\overline{b_n}),\\
&\hspace{0.8cm}[X_{n+1},x_{n+1}]\mapsto
[G_{n+1}(X_{n+1}),\overline{x_{n+1}}^{'}]
\end{align*}
For any object $(A_n,a_n)$ of $\cH_{n}(\cA_{\cdot})$, let
$Y_{n+1}=H_{n}(A_n)$. Consider the following composition morphism
$F_n(A_n)\xrightarrow[]{(\tau_{n})_{A_n}} (M_{n+1}\circ
H_n+H_{n-1}\circ
L_n+G_n)(A_n)\xrightarrow[]{1+H_{n-1}(a_n)+1}M_{n+1}(H_n(A_n))+H_{n-1}(0)+G_n(A_n)\backsimeq
M_{n+1}(Y_{n+1})+0+G_n(A_n)\backsimeq M_{n+1}(Y_{n+1})+G_n(A_n)$ of
$\cB_{n}$. We get a morphism $[Y_{n+1}\in obj(\cB_{n+1}),
y_{n+1}:F_{n}(A_n)\rightarrow
M_{n+1}(Y_{n+1})+G_{n}(A_n)]:(F_{n}(A_n),b_n)\rightarrow
(G_{n}(A_n),\overline{b_{n}})$ of $\cH_{n}(\cB_{\cdot})$ such that
the following diagram commutes
\begin{center}
\scalebox{0.9}[0.85]{\includegraphics{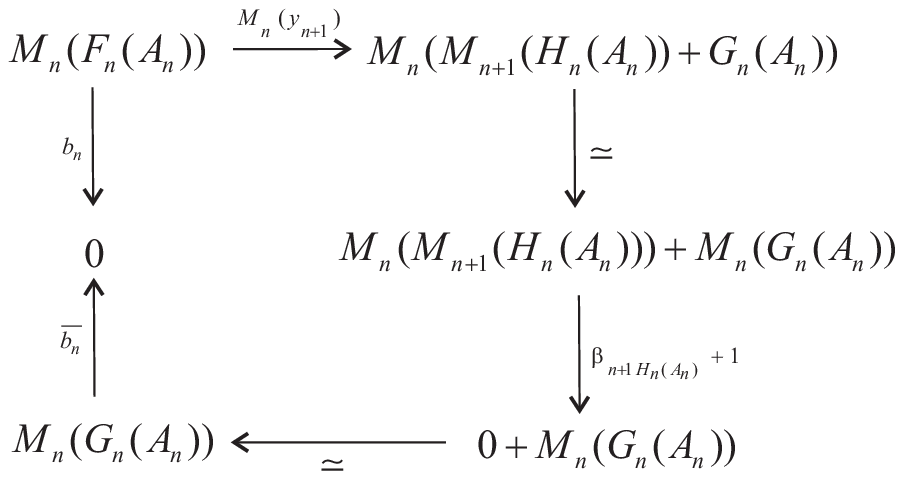}}
\end{center}
From the compatible condition of $(\tau_{n})_{n\in\mathds{Z}}$, we
have the following commutative diagram
\begin{center}
\scalebox{0.7}[0.7]{\includegraphics{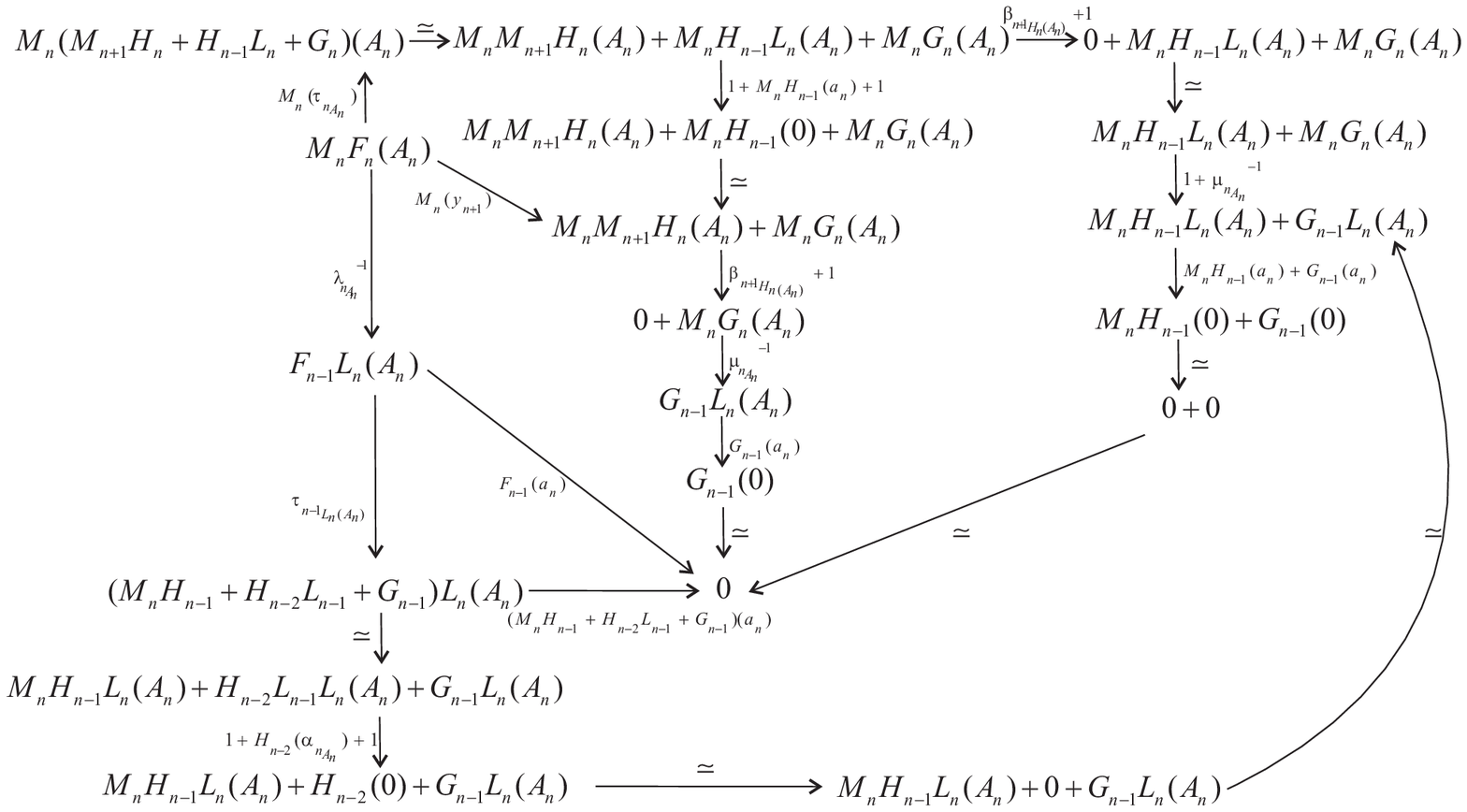}}
\end{center}
So $[Y_{n+1},y_{n+1}]$ is a morphism in $\cH_{n}(\cB_{\cdot})$, then
we can define a 2-morphism $\varphi_{n}$. For any morphism
$[X_{n+1},x_{n+1}]:(A_n,a_n)\rightarrow (A_{n}^{'},a_{n}^{'})$ in
$\cH_{n}(\cA_{\cdot})$, where $X_{n+1}\in obj(\cA_{n+1}),\
x_{n+1}:A_n\rightarrow L_{n+1}(X_{n+1})+A_{n}^{'}$ satisfying the
following commutative diagram
\begin{center}
\scalebox{0.9}[0.85]{\includegraphics{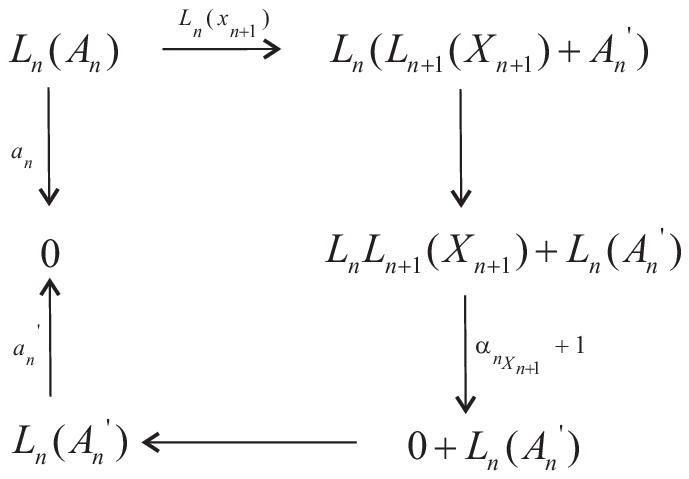}}
\end{center}
$\cH_{n}(F_{\cdot})[X_{n+1},x_{n+1}]=[F_{n+1}(X_{n+1}),\overline{x_{n+1}}]$,
$\cH_{n}(G_{\cdot})[X_{n+1},x_{n+1}]=[G_{n+1}(X_{n+1}),\overline{x_{n+1}}^{'}]$,
where $\overline{x_{n+1}}$ and $\overline{x_{n+1}}^{'}$ are the
following composition morphisms
$\overline{x_{n+1}}:F_{n}(A_n)\xrightarrow[]{F_{n}(x_{n+1})}
F_{n}(L_{n+1}(X_{n+1})+A_{n}^{'})\backsimeq (F_n\circ
L_{n+1})(X_{n+1})+F_n(A_{n}^{'})\xrightarrow[]{(\lambda_{n+1})_{X_{n+1}}+1}M_{n+1}(F_{n+1}(X_{n+1}))+F_{n}(A_{n}^{'})$,
$\overline{x_{n+1}}^{'}:G_{n}(A_n)\xrightarrow[]{G_{n}(x_{n+1})}
G_{n}(L_{n+1}(X_{n+1})+A_{n}^{'})\backsimeq (G_n\circ
L_{n+1})(X_{n+1})+G_n(A_{n}^{'})\xrightarrow[]{(\mu_{n+1})_{X_{n+1}}+1}M_{n+1}(G_{n+1}(X_{n+1}))+G_{n}(A_{n}^{'}).$

Then we have the following commutative diagram
\begin{center}
\scalebox{0.9}[0.85]{\includegraphics{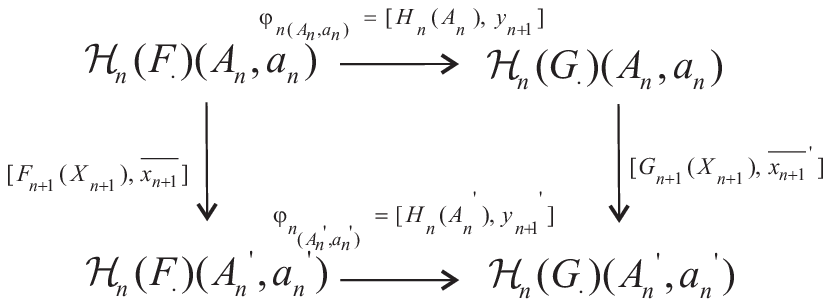}}
\end{center}
There exist $[Y_{n+2}\triangleq
H_{n+1}(X_{n+1}),y_{n+2}]:([H_{n}(A_{n}^{'}),y_{n+1}^{'}]\circ
[F_{n+1}(X_{n+1}),\overline{x_{n+1}}])\rightarrow
[G_{n+1}(X_{n+1}),\overline{x_{n+1}}^{'}]\circ[H_{n}(A_{n}),y_{n+1}]$
induced by $\tau_{\cdot}.$ In fact,
$((H_{n}(A_{n}^{'}),y_{n+1}^{'})\circ
F_{n+1}(X_{n+1}),\overline{x_{n+1}}))=[F_{n+1}(X_{n+1})+H_{n}(A_{n}^{'}),(1+y_{n+1}^{'})\circ
\overline{x_{n+1}}^{'}]$,
$[G_{n+1}(X_{n+1}),\overline{x_{n+1}}]\circ[H_{n}(A_{n}),y_{n+1}]=[H_{n}(A_n)+G_{n+1}(X_{n+1}),(1+\overline{x_{n+1}}^{'})\circ
y_{n+1}]$ from the composition of morphisms in relative cokernel, so
$y_{n+2}$ is the composition morphism
$F_{n+1}(X_{n+1})+H_{n}(A_{n}^{'})\xrightarrow[]{(\tau_{n+1})_{X_{n+1}}+1}(M_{n+2}H_{n+1}+H_{n}L_{n+1}+G_{n+1})(X_{n+1})+H_{n}(A_{n}^{'})\backsimeq
M_{n+2}H_{n+1}(X_{n+1})+H_{n}L_{n+1}(X_{n+1})+G_{n+1}(X_{n+1})+H_{n}(A_{n}^{'})\backsimeq
M_{n+2}H_{n+1}(X_{n+1})+G_{n+1}(X_{n+1})+H_{n}(L_{n+1}(X_{n+1}+A_{n}^{'})\xrightarrow[]{1+1+H_{n}(x_{n+1}^{-1})}
M_{n+2}H_{n+1}(X_{n+1})+G_{n+1}(X_{n+1})+H_{n}(A_{n})\backsimeq
M_{n+2}H_{n+1}(X_{n+1})+H_{n}(A_{n})+G_{n+1}(X_{n+1})$. Moreover the
morphism $[Y_{n+2},y_{n+2}]$ makes the following diagram commute
\begin{center}
\scalebox{0.8}[0.8]{\includegraphics{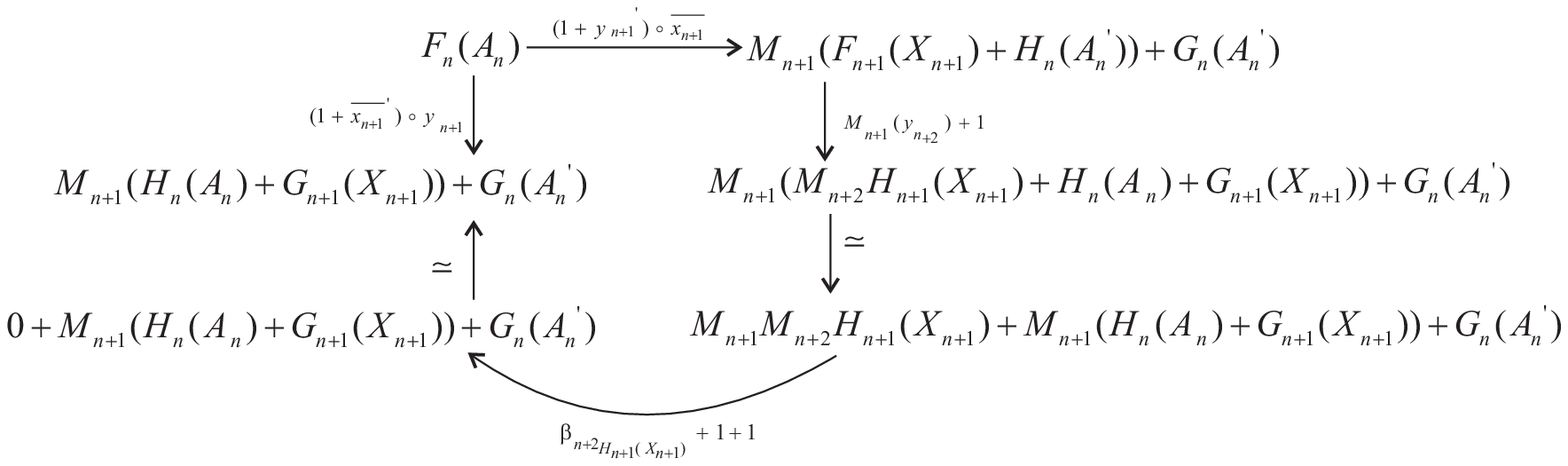}}
\end{center}
for the following several commutative diagrams
\begin{center}
\scalebox{0.8}[0.8]{\includegraphics{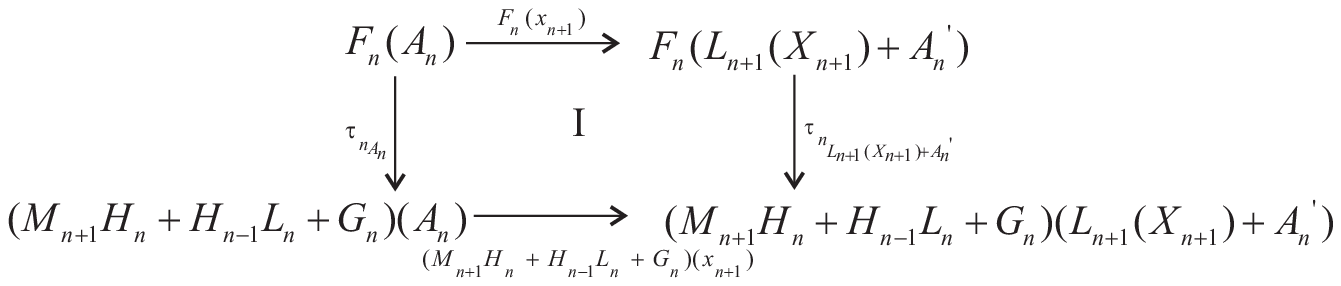}}
\end{center}
\begin{center}
\scalebox{0.8}[0.8]{\includegraphics{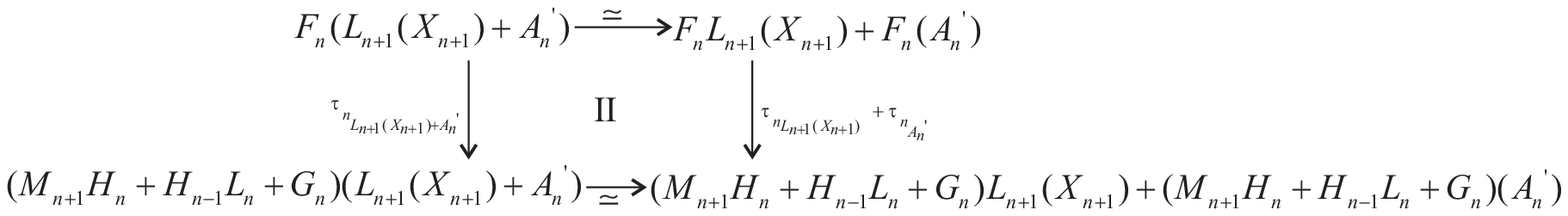}}
\end{center}
\begin{center}
\scalebox{0.8}[0.8]{\includegraphics{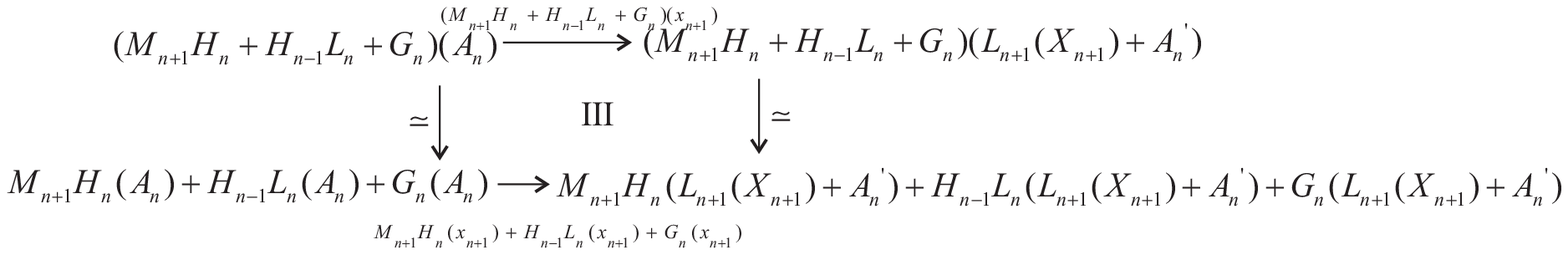}}
\end{center}
\begin{center}
\scalebox{0.7}[0.7]{\includegraphics{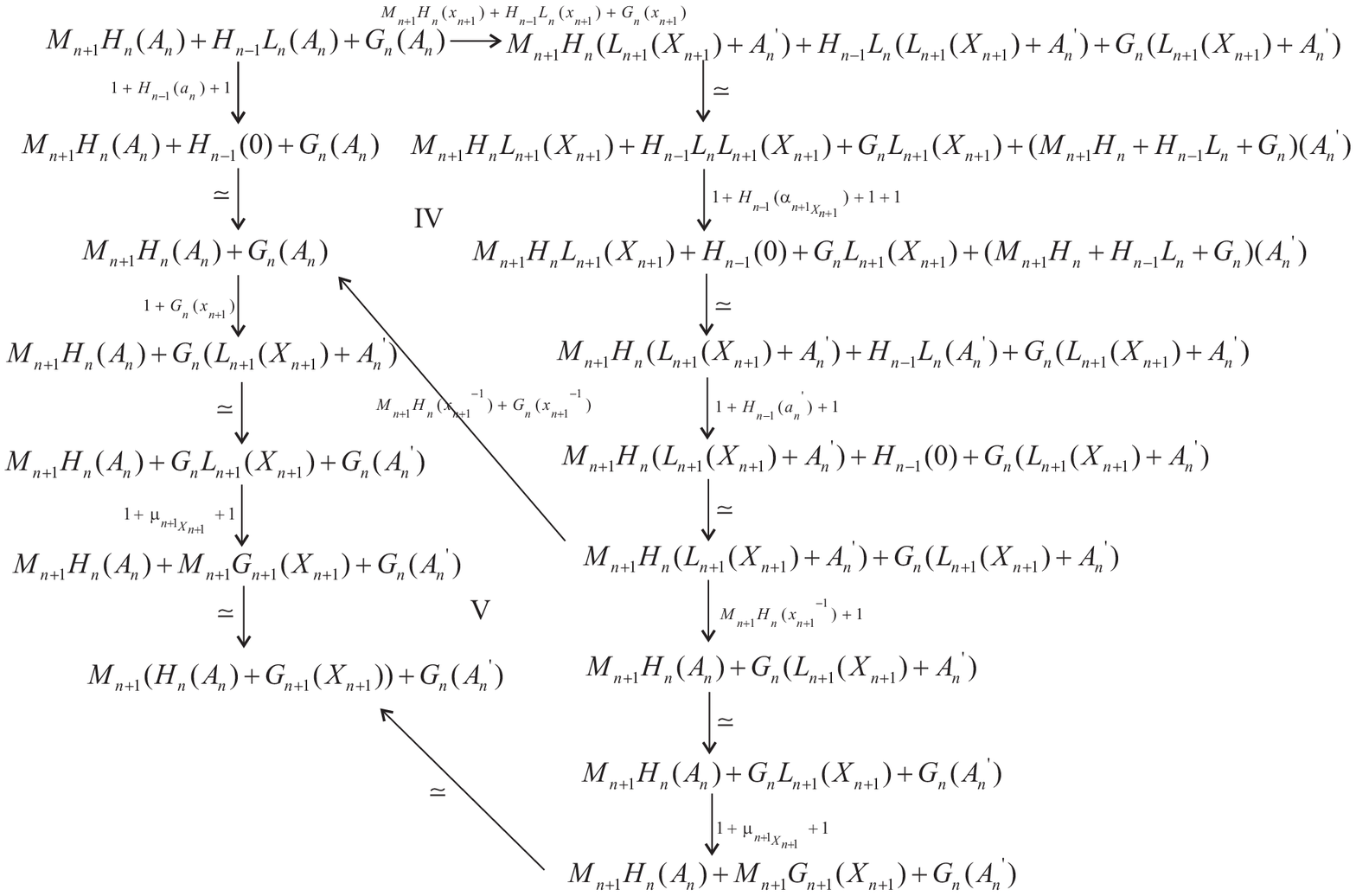}}
\end{center}
\begin{center}
\scalebox{0.7}[0.7]{\includegraphics{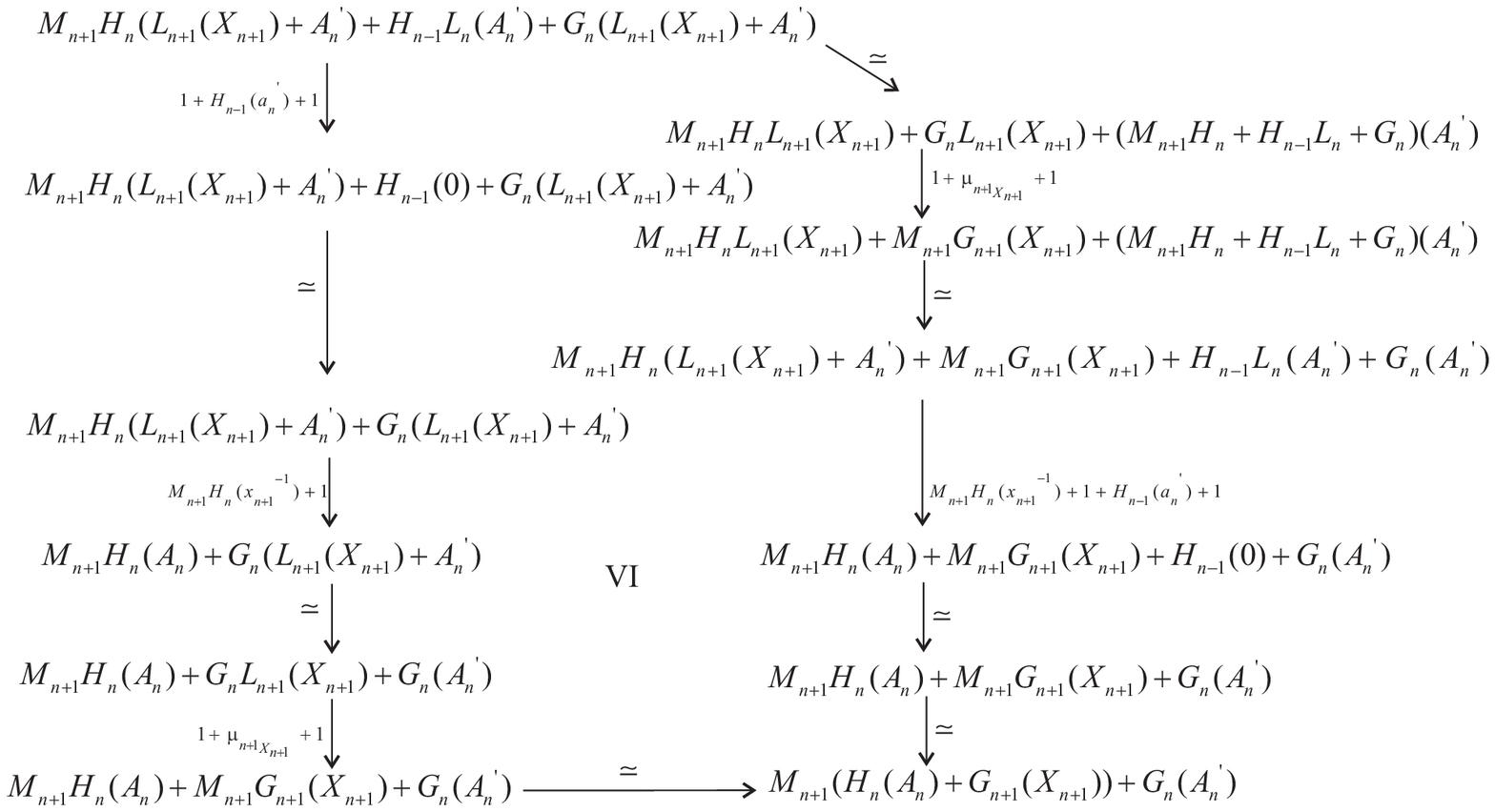}}
\end{center}
\begin{center}
\scalebox{0.7}[0.7]{\includegraphics{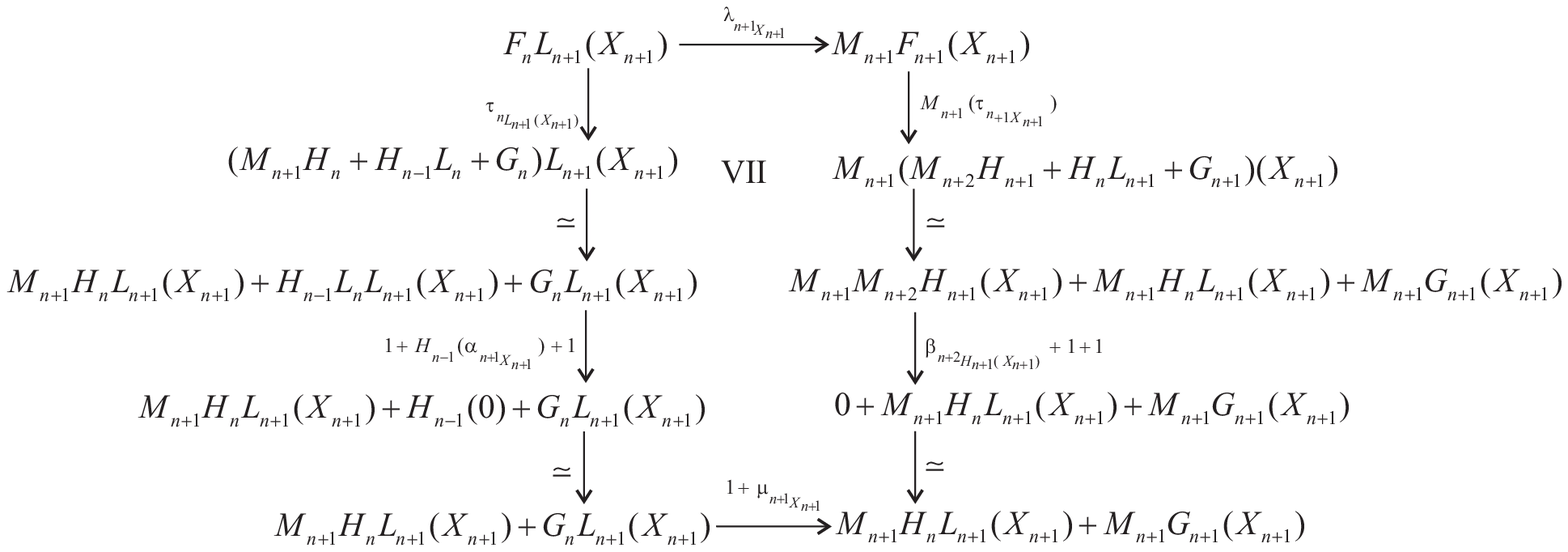}}
\end{center}
where I is commutative because $\tau_n$ is a natural transformation.
II, III follow from the properties of symmetric monoidal functors.
IV follows from operation of $H_{n-1}$ on the commutative diagram of
$[X_{n+1},x_{n+1}]$. V and VI follow from the properties of
symmetric 2-groups. VII follows from the commtutative diagram of
$\tau_{n+1}$.

Moreover, for any two objects $(A_n,a_n),\ (A_n^{'},a_n^{'})$ of
$\cH_{n}(\cA_{\cdot})$,
$(A_n,a_n)+(A_n^{'},a_n^{'})=(A_{n}+A_{n}^{'},a_{n}+a_n^{'})$ with
$L_{n-1}(a_n+a_{n}^{'})=(\alpha_{n})_{A_n+A_n^{'}}$, we have the
following commutative diagram
\begin{center}
\scalebox{0.9}[0.85]{\includegraphics{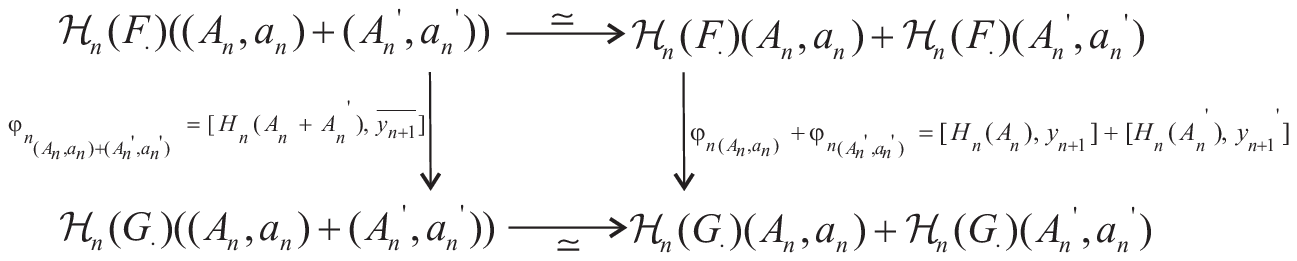}}
\end{center}
where $\overline{y_{n+1}}, y_{n+1},y_{n+1}^{'}$ are induced by
$\tau_{n}$ as above. In fact,
$[H_{n}(A_n),y_{n+1}]+[H_n(A_{n}^{'}),y_{n+1}^{'}]=[H_{n}(A_n)+H_n(A_{n}^{'}),y_{n+1}+y_{n+1}^{'}]$,
so
$(\varphi_{n})_{(A_n,a_n)+(A_n^{'},a_{n}^{'})}=(\varphi_{n})_{(A_n,a_n)}+(\varphi_{n})_{(A_n^{'},a_n^{'})}$,
then the above diagram commutes.

Then from above, we proved $\varphi_n$ is a 2-morphism in (2-SGp),
for each $n$.
\end{proof}

From the definition of 2-functors, we have the following Lemma.
\begin{Lem}Let $T:$(2-SGp)$\rightarrow$(2-SGp) be a 2-functor, $(F_{\cdot},\lambda_{\cdot}),
(G_{\cdot},\mu_{\cdot}):(\cA_{\cdot},L_{\cdot},\alpha_{\cdot})\rightarrow
(\cB_{\cdot},M_{\cdot},\beta_{\cdot})$ be two 2-chain homotopy
morphisms of complexes in (2-SGp). Then
$T(F_{\cdot},\lambda_{\cdot})$ is 2-chain homotopic to
$T(G_{\cdot},\mu_{\cdot})$ in (2-SGp).
\end{Lem}

\section{Projective resolution of symmetric 2-groups}

In this section we will give the construction of projective
resolution of any symmetric 2-group.

\begin{Def} Let $\cM$ be a symmetric 2-group.
A projective resolution of $\cM$ in (2-SGp) is a 2-chain complex of
symmetric 2-groups which is relative 2-exact in each point as in the
following diagram
\begin{center}
\scalebox{0.9}[0.85]{\includegraphics{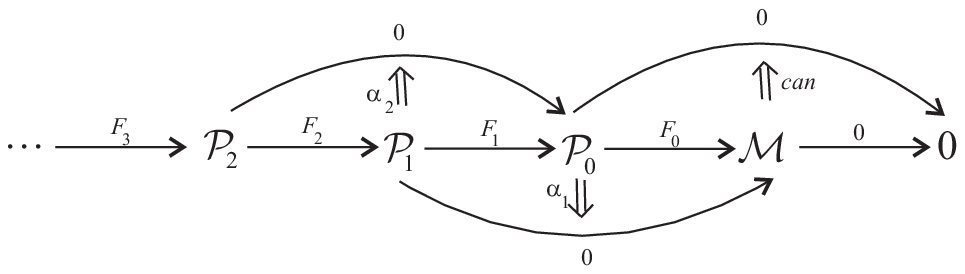}}
\end{center}
with $\cP_{n}(n\geq 0)$ projective objects in (2-SGp). i.e. the
above complex is relative 2-exact in each $\cP_{i}$ and $\cM$.
\end{Def}

\begin{Prop}
Every symmetric 2-group $\cM$ has a projective resolution in
(2-SGp).
\end{Prop}
\begin{proof}
We will construct the projective resolution of $\cM$ using the
relative kernel.

For $\cM$, there is an essentially surjective morphism
$F_0:\cP_{0}\rightarrow\cM$, with $\cP_{0}$ projective object in
(2-SGp)(\cite{14,20}). Then we get a sequence as follows
\begin{center}
\scalebox{0.9}[0.85]{\includegraphics{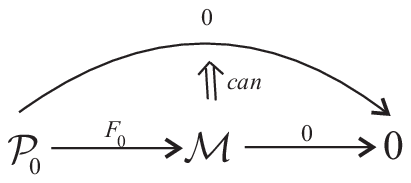}} {\footnotesize S.1.}
\end{center}
where $0:\cM\rightarrow 0$ is the zero morphism\cite{2,6}in (2-SGp),
$0$ is the symmetric 2-group with only one object and one morphism.,
$can$ is the canonical 2-morphism in (2-SGp), which is given by the
identity morphism of only one object of $0$.

From the existence of the relative kernel in (2-SGp), we have the
relative kernel $(Ker(F_0,can),\ e_{(F_0,can)},\
\varepsilon_{(F_0,can)})$
 of the sequence S.1, which is in fact the general kernel
 $(KerF_0,e_{F_0},\varepsilon_{F_0})$ \cite{11}. For the symmetric
 2-group $KerF_0$, there exists an essentially surjective morphism
 $G_1:\cP_{1}\rightarrow KerF_0$, with $\cP_{1}$ projective object in
(2-SGp)(\cite{14,20}). Let $F_{1}=e_{F_{0}}\circ
 G_{1}:\cP_{1}\rightarrow\cP_{0}$. Then we get the following
 sequence
\begin{center}
\scalebox{0.9}[0.85]{\includegraphics{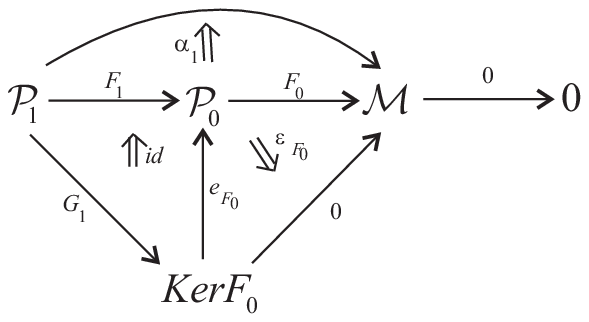}}
\end{center}
where $\alpha_{1}$ is the composition $F_{0}\circ F_{1}=F_{0}\circ
e_{F_0}\circ G_{1}\Rightarrow 0\circ G_1\Rightarrow 0$ and
compatible with $can$.

Consider the above sequence, there exists the relative kernel
$(Ker(F_1,\alpha_1),\\e_{(F_1,\alpha_1)},\varepsilon_{(F_1,\alpha_1)})$
in (2-SGp). For the symmetric 2-group $Ker(F_{1},\alpha_1)$, there
is an essentially surjective morphism $G_{2}:\cP_{2}\rightarrow
Ker(F_1,\alpha_1)$, with $\cP_{2}$ projective object in
(2-SGp)(\cite{14,20}). Let $F_2=e_{(F_1,\alpha_1)}\circ
G_2:\cP_2\rightarrow\cP_1$. Then we get a sequence
\begin{center}
\scalebox{0.9}[0.85]{\includegraphics{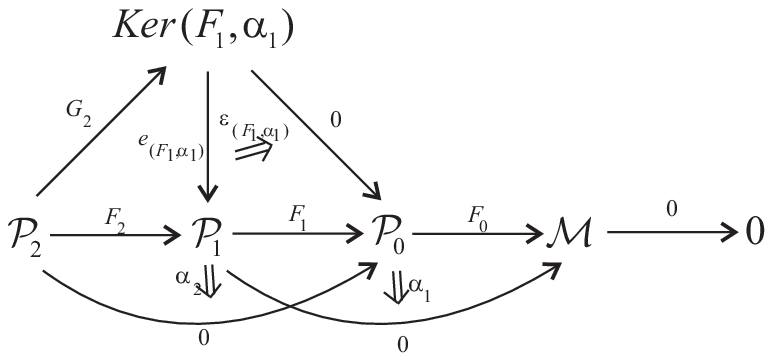}}
\end{center}
where $\alpha_2$ is the composition $F_{1}\circ F_{2}=F_{1}\circ
e_{(F_1,\alpha_1)}\circ G_{2}\Rightarrow 0\circ G_2\Rightarrow 0$
and compatible with $\alpha_1$.

Using the same method, we get a 2-chian complex of symmetric
2-groups
\begin{center}
\scalebox{0.9}[0.85]{\includegraphics{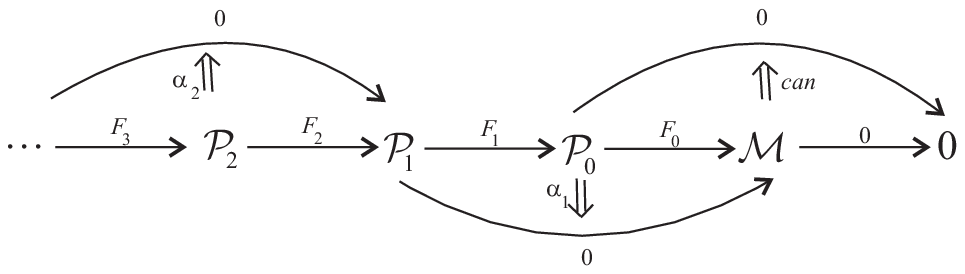}} {\footnotesize
Fig.2.}
\end{center}
Next, we will check that the complex Fig.2 is relative 2-exact in
each point.

Firstly, the complex Fig.2 is relative 2-exact in $\cM$. In fact,
$F_0$ is essentially surjective(\cite{11}).

Secondly, the complex Fig.2 is relative 2-exact in $\cP_0$. From the
cancellation property of $e_{F_0}$, there exists
$\overline{\alpha_2}:G_1\circ F_2\Rightarrow 0$ defined by
$\overline{(\alpha_{2})_{y}}\triangleq
(\alpha_{2})_{y}:G_{1}F_{2}(y)\rightarrow 0,\ \forall y\in
obj(\cP_{2})$. And $G_{1}:\cP_{1}\rightarrow KerF_0$ is in fact
$G_{1}(x)=(F_{1}(x),(\alpha_{1})_{x})$. For any $x_1,x_2\in
obj(\cP_1)$ and the morphism $g:G_{1}(x_1)\rightarrow G_{1}(x_2)$ of
$KerF_0$. Under the morphism $e_{F_0}:KerF_0\rightarrow \cP_0$, we
have a morphism $e_{F_0}(g):e_{F_0}G_{1}(x_1)=F_{1}(x_1)\rightarrow
e_{F_0}G_{1}(x_2)=F_{1}(x_2)$ of $\cP_0$, then we get a composition
morphism $e_{F_0}(g)+1:F_{1}(x_1+x_{2}^{*})\backsimeq
F_{1}(x_1)+F_{1}(x_2)^{*}\rightarrow
F_{1}(x_2)+F_{1}(x_2)^{*}\backsimeq 0$ and a commutative diagram
\begin{center}
\scalebox{0.9}[0.85]{\includegraphics{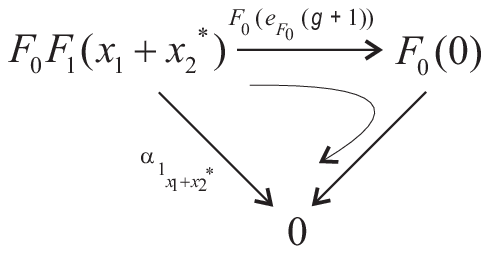}}
\end{center}
by the compatibility of $\varepsilon_{F_0}$ and $\alpha_1$.

We get an object $(x_{1}+x_{2}^{*},e_{F_0}(g)+1)$ in
$Ker(F_1,\alpha_1)$, and from the essentially surjective morphism
$G_2:\cP_2\rightarrow Ker(F_1,\alpha_1)$, there exist an object $y$
in $\cP_2$ and the isomorphism
$h:(x_1+x_{2}^{*},e_{F_0}(g)+1))\rightarrow G_{2}(y)$ in
$Ker(F_1,\alpha_1)$. Using the morphism
$e_{(F_1,\alpha_1)}:Ker(F_1,\alpha_1)\rightarrow \cP_1$, we have a
morphism $e_{(F_1,\alpha_1)}(h):x_1+x_{2}^{*}\rightarrow
e_{(F_1,\alpha_1)}G_{2}(y)=F_{2}(y)$. So we have a morphism
$$
f:x_1\rightarrow x+0\rightarrow x_1+(x_{2}^{*}+x_{2})\rightarrow
(x_1+x_{2}^{*})+x_{2}\rightarrow F_{2}(y)+x_2,
$$
such that
\begin{center}
\scalebox{0.9}[0.85]{\includegraphics{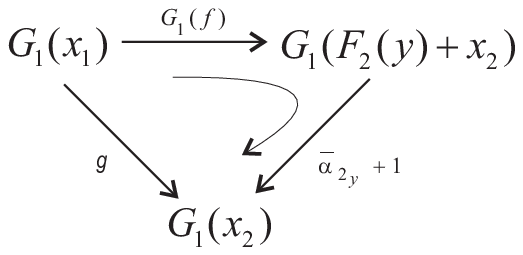}}
\end{center}
So we proved that the essentially surjective morphism $G_1$ is
$\overline{\alpha_2}$-full, the complex Fig.2 is relative 2-exact in
$\cP_0$.

Using the same method, we can prove the complex Fig.2 is relative
2-exact in each point.
\end{proof}

\begin{Thm}
Let $(F_{\cdot}:\cP_{\cdot}\rightarrow \cM,\alpha_{\cdot})$ be a
projective resolution of symmetric 2-group $\cM$, and
$H:\cM\rightarrow \cN$ a morphism in (2-SGp). Then for any
projective resolution $(G_{\cdot}:\cQ_{\cdot}\rightarrow
\cN,\beta_{\cdot})$, there is a morphism
$H_{\cdot}:\cP_{\cdot}\rightarrow \cQ_{\cdot}$ of complexes in
(2-SGp) together with the family of 2-morphisms
$\{\varepsilon_{n}:G_{n}\circ H_{n}\Rightarrow H_{n-1}\circ
F_{n}\}_{n\geq 0}$(where $H_{-1}=H$) as in the following diagram
\begin{center}
\scalebox{0.9}[0.85]{\includegraphics{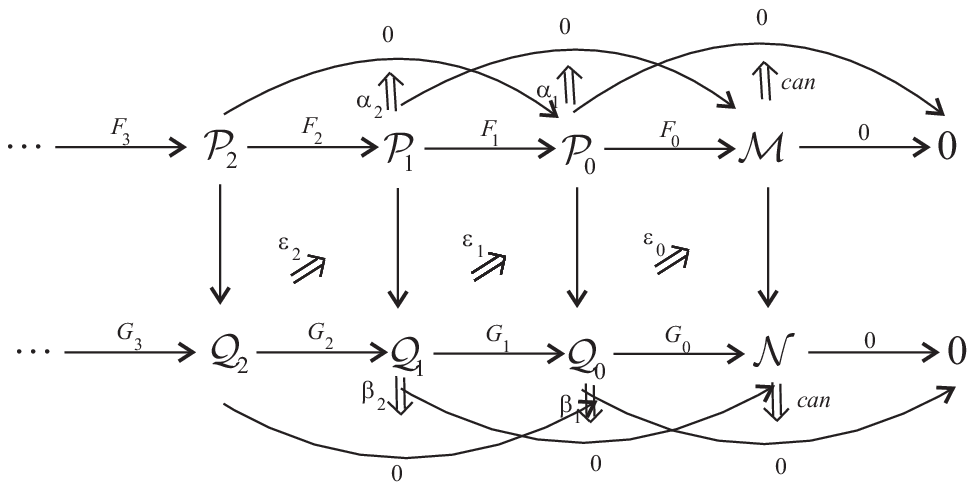}}
\end{center}
If there is another morphism between projective resolutions, they
are 2-chain homotopy.
\end{Thm}
\begin{proof}
The existence of $H_{0}:\cP_{0}\rightarrow \cQ_{0}$: Since $G_{0}$
is essentially surjective and $\cP_0$ is a projective object in
(2-SGp), there exist 1-morphism $H_0:\cP_{0}\rightarrow \cQ_{0}$ and
2-morphism $\varepsilon_{0}:G_{0}\circ H_{0}\Rightarrow H\circ
F_{0}$ as follows
\begin{center}
\scalebox{0.9}[0.85]{\includegraphics{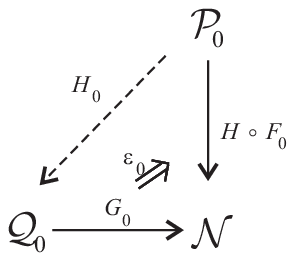}}
\end{center}

Consider the morphism $H_0:\cP_{0}\rightarrow \cQ_{0}$, we have a
morphism
\begin{align*}
&\hspace{1.1cm}\overline{H_{0}}:KerF_{0}\rightarrow KerG_{0}\\
&\hspace{1.8cm}(x_0,a_0)\mapsto (H_{0}(x_0),\widetilde{a_0}),\\
&(x_0,a_0)\xrightarrow[]{f_0}(x_{0}^{'},a_{0}^{'})\mapsto
(H_0(x_0),\widetilde{a_{0}})\xrightarrow[]{H_{0}(f_0)}(H_0(x_{0}^{'}),\widetilde{a_{0}}^{'})
\end{align*}
where $\widetilde{a_0}$ is the composition $(G_{0}\circ
H_{0})(x_0)\xrightarrow[]{(\varepsilon_{0})_{x_0}}(H\circ
F_{0})(x_0)\xrightarrow[]{H(a_0)} H(0)\backsimeq 0$. Moreover, there
is a commutative diagram
\begin{center}
\scalebox{0.9}[0.85]{\includegraphics{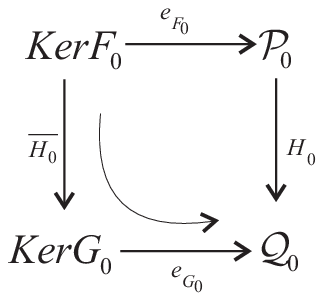}}
\end{center}
From the relative 2-exactness of projective resolution of symmetric
2-group,
there exist essentially surjective morphisms
$\overline{F_{1}}:\cP_{1}\rightarrow kerF_{0}$,
$\overline{G_{1}}:\cQ_{1}\rightarrow kerG_{0}$ and 2-morphisms
$\varphi_{1}:e_{F_0}\circ \overline{F_1}\Rightarrow F_1$,
$\psi_{1}:e_{G_0}\circ \overline{G_{1}}\Rightarrow G_{1}$,
respectively. Then there exist 1-morphism $H_1:\cP_1\rightarrow
\cQ_1$ and 2-morphism
$\overline{\varepsilon_{1}}:\overline{G_1}\circ H_1\Rightarrow
\overline{H_0}\circ\overline{F_1}$.
From $\overline{\varepsilon_{1}}$ and $e_{G_0}\circ
\overline{H_0}=H_0\circ e_{F_0}$, we can define a 2-morphism
$\varepsilon_1:G_1\circ H_1\Rightarrow H_0\circ F_1$ by
$(\varepsilon_{1})_{x_1}: (G_1\circ
H_1)(x_1)\xrightarrow[]{(\psi_{1})_{H_{1}(x_{1})}^{-1}}e_{G_0}\circ
\overline{G_1}\circ
H_{1}(x_1)\xrightarrow[]{e_{G_0}(\overline{(\varepsilon_{1})}_{x_1})}e_{G_0}\circ\overline{H_{0}}\circ\overline{F_1}(x_1)=H_0\circ
e_{F_0}\circ
\overline{F_1}(x_1)\xrightarrow[]{H_{0}((\varphi_{1})_{x_{1}})}H_{0}\circ
F_{1}(x_1)$, which is compatible with $\varepsilon_0 $.

Next we will construct $H_{n}$ and $\varepsilon_{n}:G_{n}\circ
H_{n}\Rightarrow H_{n-1}\circ F_{n}$ by induction on $n$.
Inductively, suppose $H_{i}$ and $\varepsilon_{i}$ have been
constructed for $i\leq n$ satisfying the compatible conditions.
Consider the morphism $H_{n-1}:\cP_{n-1}\rightarrow\cQ_{n-1}$, there
is an induced morphism
\begin{align*}
&\overline{H_{n-1}}:Ker(F_{n-1},\alpha_{n-1})\rightarrow
Ker(G_{n-1},\beta_{n-1})\\
&\hspace{1.8cm}(x_{n-1},a_{n-1})\mapsto(H_{n-1}(x_{n-1}),\widetilde{a_{n-1}}),\\
&\hspace{3.2cm}f_{n-1}\mapsto H_{n-1}(f_{n-1})
\end{align*}
where $\widetilde{a_{n-1}}$ is the composition $G_{n-1}\circ
H_{n-1}(x_{n-1})\xrightarrow[]{}H_{n-2}\circ
F_{n-1}(x_{n-1})\xrightarrow[]{H_{n-2}(a_{n-1})}H_{n-1}(0)\simeq 0$.
Moreover, there is the following commutative diagram
\begin{center}
\scalebox{0.9}[0.85]{\includegraphics{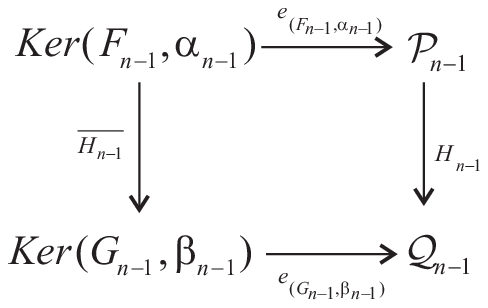}}
\end{center}

Using the relative 2-exactness of projective resolutions of $\cM$
and $\cN$, we have the following diagram
\begin{center}
\scalebox{0.9}[0.85]{\includegraphics{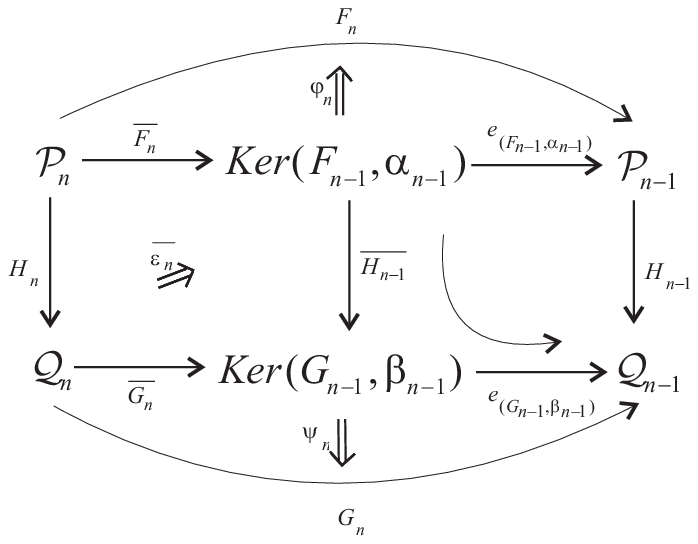}}
\end{center}
The existence of $H_{n}$ and $\overline{\varepsilon_{n}}$ come from
the projectivity of $\cP_{n}$. Similar to the appearing of
$\varepsilon_{1}$, there is a 2-morphism $\varepsilon_{n}$ given by
$\overline{\varepsilon_{n}}$, compatible with $\varepsilon_{n-1}$.


Next, we show the uniqueness of $(H_{\cdot},\varepsilon_{\cdot})$ up
to 2-chain homotopy. Suppose $(K_{\cdot},\zeta_{\cdot})$ is another
morphism of projective resolutions. We will construct the 1-morphism
$T_{n}:\cP_n\rightarrow \cQ_{n+1},$ and 2-morphism $
\tau_{n}:H_n\Rightarrow G_{n+1}\circ T_{n}+T_{n-1}\circ F_n+K_n$ by
induction on $n$. If $n<0$, $\cP_{n}=0$, so we get $T_n=0$. If
$n=0$, there is a 1-morphism $H_0-K_0:\cP_0\rightarrow KerG_0$,
together with essentially surjective morphism
$\overline{G_1}:\cQ_1\rightarrow KerG_0$, there exist a morphism
$T_0:\cP_0\rightarrow \cQ_1$ and 2-morphism
$\tau_{0}^{'}:\overline{G_1}\circ T_0\Rightarrow H_0-K_0$. Then we
get a 2-morphism $\tau_{0}:H_0\Rightarrow G_1\circ T_0+K_0$.

Inductively, we suppose given family of morphisms
$(H_i,\tau_i)_{i\leq n}$ so that
$H_i:\cP_i\rightarrow\cQ_{i+1},\tau_{i}:H_i\Rightarrow G_{i+1}\circ
T_{i}+T_{i-1}\circ F_i+K_{i}.$ Consider the 1-morphism
$H_n-K_n-T_{n-1}\circ F_n:\cP_n\rightarrow Ker(G_n,\beta_n)$ and
essentially surjective morphism
$\overline{G_{n+1}}:\cQ_{n+1}\rightarrow Ker(G_n,\beta_n)$, there
exist a 1-morphism $T_n:\cP_n\rightarrow\cQ_{n+1}$ and a 2-morphism
$\tau_{n}^{'}:\overline{G_{n+1}}\circ T_n\Rightarrow
H_n-K_n-T_{n-1}\circ F_n$. Then we get a 2-morphism
$\tau:H_n\Rightarrow G_{n+1}\circ T_n+T_{n-1}\circ F_n+K_n$.
\end{proof}

\section{Derived 2-Functor}
In this section, we will give the left derived 2-functor in the
abelian 2-category (2-SGp), which has enough projective
objects\cite{14,20}.

\begin{Def}
An additive 2-functor(\cite{2}) $T$: (2-SGp)$\rightarrow$(2-SGp) is
called right relative 2-exact if the relative 2-exactness of
\begin{center}
\scalebox{0.9}[0.85]{\includegraphics{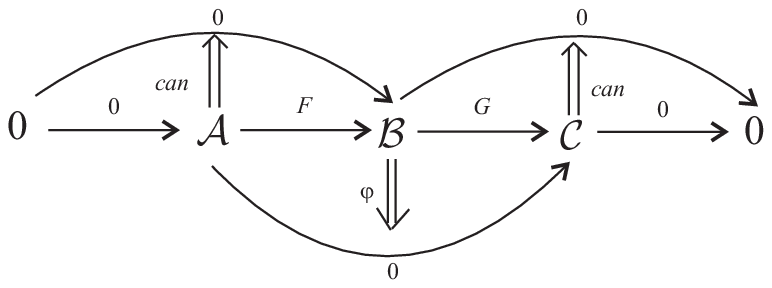}}
\end{center}
in $\cA,\cB$ and $\cC$ implies relative 2-exactness of
\begin{center}
\scalebox{0.9}[0.85]{\includegraphics{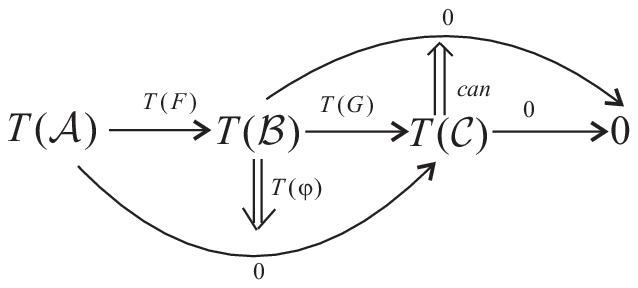}}
\end{center}
in $T(\cB)$ and $T(\cC)$.
\end{Def}
The left relative 2-exact 2-functor can be defined dually.

By Remark 2 and Proposition 1, Theorem 1, there is
\begin{cor}
Let $T: $(2-SGp)$\rightarrow$(2-SGp) be an additive 2-functor,
and $\cA$ be any object of (2-SGp). For two projective resolutions
$\cP_{\cdot},\ \cQ_{\cdot}$ of $\cA$, there is an equivalence
between homology symmetric 2-groups $\cH_{\cdot}(T(\cP_{\cdot}))$
and $\cH_{\cdot}(T(\cQ_{\cdot}))$.
\end{cor}

Let $T$: (2-SGp)$\rightarrow$(2-SGp) be an additive
2-functor. There is a 2-functor
\begin{align*}
&\cL_{i}T:\textrm{(2-SGp)}\rightarrow \textrm{(2-SGp)}\\
&\hspace{2.2cm}\cA\mapsto \cL_{i}T(\cA),\\
&\hspace{1.3cm}\cA\xrightarrow[]{F}\cB\mapsto
\cL_{i}T(\cA)\xrightarrow[]{\cL_{i}T(F)} \cL_{i}T\cB),
\end{align*}
where $\cL_{i}T(\cA)$ is defined by $\cH_{i}(T(\cP_{\cdot}))$, and
$\cP_{\cdot}$ is the projective resolution  of $\cA$. $\cL_{i}T$ is
a well-defined 2-functor from the properties of additive 2-functor
and Corollary 1.
\begin{cor}
Let $T$: (2-SGp)$\rightarrow$(2-SGp) be a right relative 2-exact
2-functor, and $\cA$ be a projective object in (2-SGp). Then
$\cL_{i}T(\cA)=0$ for $i\neq0$.
\end{cor}

The following is a basic property of derived functors.
\begin{Thm}
The left derived 2-functor $\cL_{*}T$ takes
the sequence of symmetric 2-groups
\begin{center}
\scalebox{0.9}[0.85]{\includegraphics{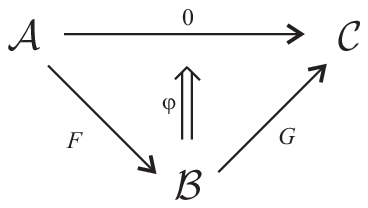}}
\end{center}
which is relative 2-exact in $\cA,\ \cB,\ \cC$ to a long sequence
2-exact(\cite{1,6})in each point
\begin{center}
\scalebox{0.9}[0.85]{\includegraphics{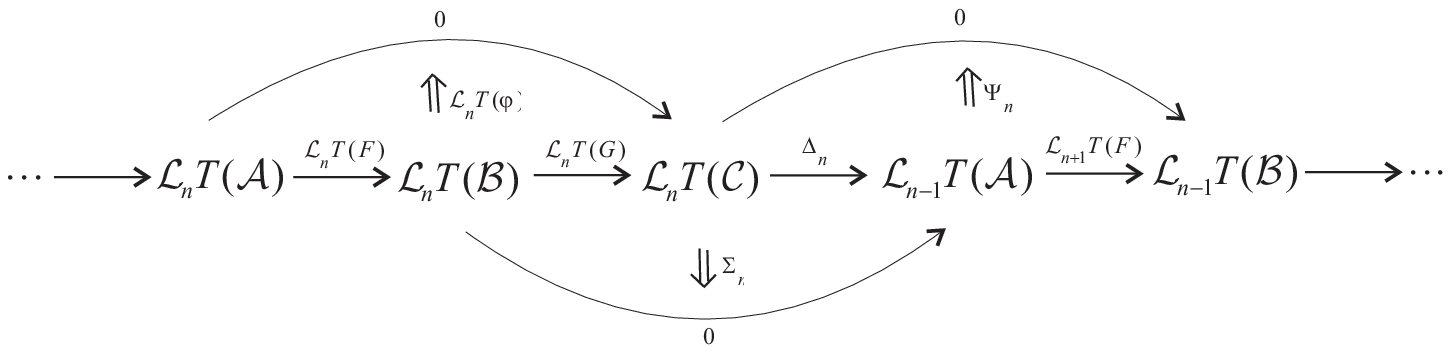}}
\end{center}

\end{Thm}
In order to prove this theorem, we need the following Lemmas.
\begin{Lem}
Let $\cP$ and $\cQ$ be projective objects in (2-SGp). Then the
product category $\cP\times\cQ$ is a projective object in (2-SGp).
\end{Lem}
\begin{proof}
First we know that $\cP\times\cQ$ is a symmetric
2-group(\cite{2,4}). So we need to prove the projectivity of it.
There are canonical morphisms
$$
\cP\xleftarrow[]{p_1}\cP\times\cQ\xrightarrow[]{p_2}\cQ,\
\cP\xrightarrow[]{i_1}\cP\times\cQ\xleftarrow[]{i_2}\cQ.
$$
For any morphism $G:\cP\times\cQ\rightarrow \cB$, there are
composition morphisms
$G_1:\cP\xrightarrow[]{i_1}\cP\times\cQ\xrightarrow[]{G}\cB$,
$G_2:\cQ\xrightarrow[]{i_2}\cP\times\cQ\xrightarrow[]{G}\cB$. Then
for an essentially surjective functor $F:\cA\rightarrow \cB$, there
exist 1-morphisms $G_{1}^{'}:\cP\rightarrow \cA$,
$G_{2}^{'}:\cQ\rightarrow \cA$ and 2-morphisms $h_{1}:F\circ
G_{1}^{'}\Rightarrow G_{1}$, $h_{2}:F\circ G_{2}^{'}\Rightarrow
G_{2}$ since $\cP$ and $\cQ$ are projective objects in (2-SGp).

So there are 1-morphism $G^{'}:\cP\times\cQ\rightarrow\cA$ given by
$G^{'}\triangleq G_{1}^{'}\circ p_{1}+G_{2}^{'}\circ p_{2}$ and
2-morphism $h:F\circ G^{'}\Rightarrow G:\cP\times\cQ\rightarrow\cB$
given by the composition $h_{(x,y)}:(F\circ
G^{'})(x,y)=F(G_{1}^{'}(x)+G_{2}^{'}(y))\backsimeq
F(G_{1}^{'}(x))+F(G_{2}^{'}(y))\xrightarrow[]{(h_{1})_{x}+(h_{2})_{y}}G_{1}(x)+G_{2}(y)=G(x,0)+G(0,y)\backsimeq
G((x,0)+(0,y))=G(x,y)$, for any $(x,y)\in obj(\cP\times\cQ)$.

Then $\cP\times\cQ$ is a projective object in (2-SGp).
\end{proof}

\begin{Lem}
Let $(F,\varphi,G):\cA\rightarrow\cB\rightarrow\cC$ be an extension
of symmetric 2-groups in (2-SGp)(\cite{1,11}),
$(\cP_{\cdot},L_{\cdot},\alpha_{\cdot})$
$(\cQ_{\cdot},N_{\cdot},\beta_{\cdot})$ be projective resolutions of
$\cA$ and $\cC$, respectively. Then there is a projective resolution
$(\cK_{\cdot},M_{\cdot},\varphi_{\cdot})$ of $\cB$, such that
$\cP_{\cdot}\rightarrow \cK_{\cdot}\rightarrow\cQ_{\cdot}$ forms an
extension of complexes in (2-SGp).
\end{Lem}
\begin{proof}
We give the construction of projective resolution
$(\cK_{\cdot},M_{\cdot},\varphi_{\cdot})$ of $\cB$ in the following
steps.

Step 1. Since $\cQ_0$ is a projective object, together with
essentially surjective $G:\cB\rightarrow\cC$ and 1-morphism
$N_0:\cQ_{0}\rightarrow \cC$, there exist 1-morphism
$\overline{N_0}:\cQ_0\rightarrow \cB$ and 2-morphism $h_0:G\circ
\overline{N_0}\Rightarrow N_0$. Then we can define a 1-morphism
\begin{align*}
&\hspace{2cm}M_0:\cP_{0}\times\cQ_{0}\rightarrow \cB\\
&\hspace{3cm}(x_0,y_0)\mapsto M_{0}(x_0,y_0)\triangleq
F(L_0(x_0))+\overline{N_{0}}(y_0),\\
&\hspace{3cm}
(f_0,g_0)
\mapsto FL_{0}(f_0)+\overline{N_0}(g_0).
\end{align*}
Moreover, $M_0$ is essentially surjective. In fact, for any $B\in
obj(\cB)$, we have $G(B)\in obj(\cC)$. Since $N_0:\cQ_{0}\rightarrow
\cC$ is essentially surjective, there are $y_{0}\in obj(\cQ_0)$ and
isomorphism $N_{0}(y_0)\rightarrow G(B)$, together with 2-morphism
$h_0:G\circ \overline{N_0}\Rightarrow N_0$. We get a composition
isomorphism
$G(\overline{N_0}(y_0))\xrightarrow[]{(h_{0})_{y_0}}N_{0}(y_0)\rightarrow
G(B)$. Moreover, we get an isomorphism
$c:G(B+\overline{N_0}(y_0)^{*})\rightarrow 0$ in $\cC$. Then we
obtain an object $(B+\overline{N_0}(y_0)^{*},c)$ of $KerG$. Since
$(F,\varphi,G):\cA\rightarrow\cB\rightarrow\cC$ is an extension, by
the definition of extension, there is an equivalence
$F_0:\cA\rightarrow KerG$, which is essentially surjective, so there
are $A\in obj(\cA)$ and isomorphism $F_{0}(A)\rightarrow
(B+\overline{N_0}(y_0)^{*},c)$. For $A\in obj(\cA)$ and essentially
surjective morphism $L_0:\cP_{0}\rightarrow\cA$, there are $x_0\in
obj(\cP_0)$ and isomorphism $L_{0}(x_0)\rightarrow A$. Then we get a
composition isomorphism
$$
F(L_{0}(x_0))\rightarrow F(A)\rightarrow e_{G}(F_{0}(A))\rightarrow
e_{G}((B+\overline{N_0}(y_0)^{*},c))=B+\overline{N_0}(y_0)^{*}.
$$
There is an isomorphism
$$
F(L_{0}(x_0))+\overline{N_0}(y_0)\rightarrow B.
$$
Then, for any $B\in obj(\cB)$, there are $(x_0,y_0)\in
obj(\cP_{0}\times\cQ_{0})$ and isomorphism
$M_{0}(x_0,y_0)=F(L_{0}(x_0))+\overline{N_0}(y_0)\rightarrow B$.

Also,
\begin{center}
\scalebox{0.9}[0.85]{\includegraphics{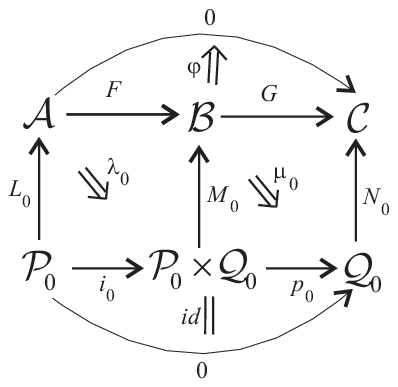}}
\end{center}
is the morphism of extensions in (2-SGp), where $\lambda_0:F\circ
L_0\Rightarrow M_{0}\circ i_{0}$ is given by
$(\lambda_{0})_{x_0}:F(L_{0}(x_{0})\backsimeq F(L_{0}(x_0))+0
\backsimeq
F(L_{0}(x_0))+\overline{N_{0}}(0)=M_{0}(x_0,0)=M_{0}(i_{0}(x_0))$,
for all $x_0\in obj(\cP_0)$. $\mu_{0}:G\circ M_0\Rightarrow
N_{0}\circ p_0$ is given by $(\mu_{0})_{(x_0,y_0)}:(G\circ
M_0)(x_0,y_0)=G(FL_{0}(x_0)+\overline{N_0}(y_0))\backsimeq
G(FL_{0}(x_0))+G(\overline{N_0}(y_0))\xrightarrow[]{\varphi_{L_{0}(x_0)}+(h_0)_{y_0}}N_0(y_0)$.

Step 2. From the definition of relative 2-exactness, there are
essentially surjective 1-morphisms $L_{1}^{'}:\cP_{1}\rightarrow
KerL_{0}$, $N_{1}^{'}:\cQ_{1}\rightarrow KerN_{0}$ as in the
following diagram
\begin{center}
\scalebox{0.9}[0.85]{\includegraphics{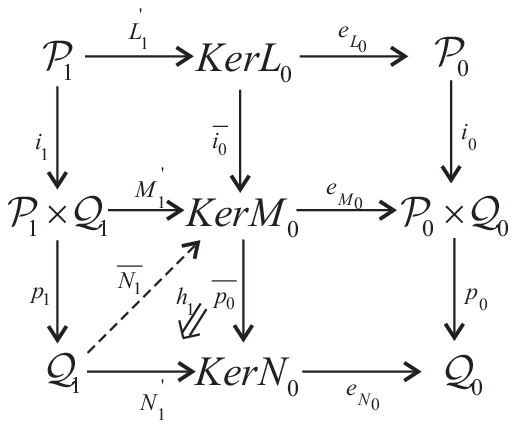}}
\end{center}
where $M_{1}^{'}:\cP_{1}\times\cQ_{1}\rightarrow KerM_{0}$ is given
by $M_{1}^{'}(x_1,y_1)\triangleq (\overline{i_0}\circ
L_{1}^{'})(x_1)+\overline{N_1}(y_1)$, for any $(x_1,y_1)\in
obj(\cP_{1}\times\cQ_{1})$, which is essentially surjective from the
proof of step 1.

Then we get a composition 1-morphism $M_{1}=e_{M_0}\circ
M_{1}^{'}:\cP_{1}\times\cQ_{1}\rightarrow \cP_{0}\times\cQ_{0}$, and
a composition 2-morphism $\varphi_{1}:M_{0}\circ M_{1}\Rightarrow
0\circ M_{1}^{'}\Rightarrow 0$, such that
\begin{center}
\scalebox{0.9}[0.85]{\includegraphics{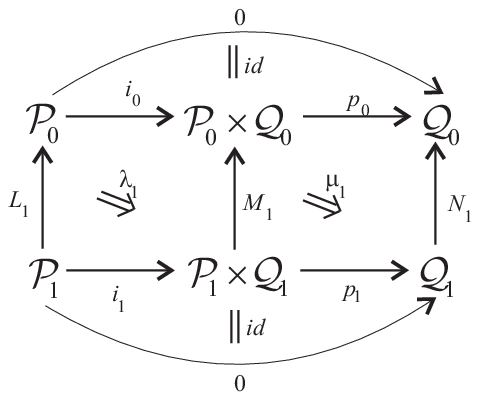}}
\end{center}
is a morphism of extensions in (2-SGp), where $\lambda_{1}$ and
$\mu_{1}$ are given in the natural way as in step 1.

Step 3. From the definition of relative 2-exactness, there are
essentially surjective 1-morphisms $L_{2}^{'}:\cP_{2}\rightarrow
Ker(L_{1},\alpha_{1})$, $N_{2}^{'}:\cQ_{2}\rightarrow
Ker(N_{1},\beta_{1})$ as in the following diagram
\begin{center}
\scalebox{0.9}[0.85]{\includegraphics{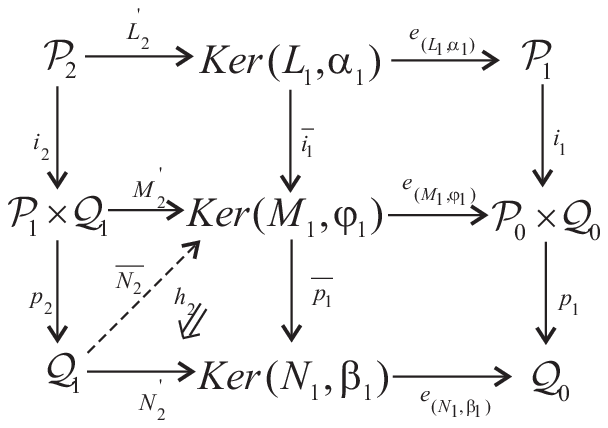}}
\end{center}
where $M_{2}^{'}:\cP_{2}\times\cQ_{2}\rightarrow
Ker(M_{1},\varphi_{1})$ is given by $M_{2}^{'}(x_2,y_2)\triangleq
(\overline{i_1}\circ L_{2}^{'})(x_2)+\overline{N_2}(y_2)$, for any
$(x_2,y_2)\in obj(\cP_{2}\times\cQ_{2})$, which is essentially
surjective from the proof of step 1.

Then we get a composition 1-morphism
$M_{2}=e_{(M_{1},\varphi_{1})}\circ
M_{2}^{'}:\cP_{2}\times\cQ_{2}\rightarrow \cP_{1}\times\cQ_{1}$, and
a composition 2-morphism $\varphi_{2}:M_{1}\circ M_{2}\Rightarrow
0\circ M_{2}^{'}\Rightarrow 0$, such that
\begin{center}
\scalebox{0.9}[0.85]{\includegraphics{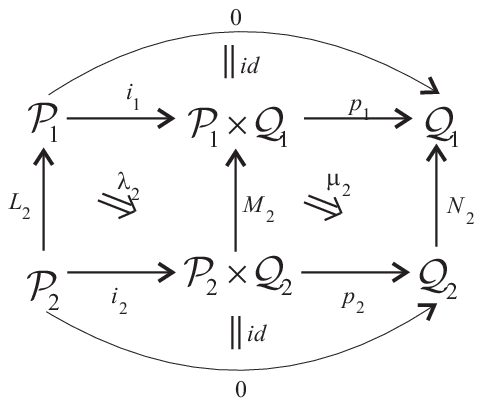}}
\end{center}
is a morphism of extensions in (2-SGp), where $\lambda_{2}$ and
$\mu_{2}$ are given in the natural way as in step 1.

Using the same method, we get a complex
$(\cP\times\cQ_{\cdot},M_{\cdot},\varphi_{\cdot})$ of product
symmetric 2-groups. Using the methods in Proposition 2, this complex
is relative 2-exact in each point, and
$(i_{\cdot},id,p_{\cdot}):\cP_{\cdot}\rightarrow
\cP_{\cdot}\times\cQ_{\cdot}\rightarrow\cQ_{\cdot}$ forms an
extension of complexes in (2-SGp).

Set $\cK_{n}=\cP_{n}\times\cQ_{n}$, for $n\geq 0$, which are
projective objects in (2-SGp) by Lemma 3.This finishes the proof.

\end{proof}

By the universal property of (bi)product of symmetric 2-groups and
the property of additive 2-functor(\cite{2}). We get
\begin{Lem}
Let $T$ be an additive 2-functor in 2-category (2-SGp), and $\cA,\
\cB$ be objects in (2-SGp). Then there is an equivalence between
$T(\cA\times\cB)$ and $T(\cA)\times T(\cB)$ in (2-SGp).
\end{Lem}

Proof of Theorem 2. For symmetric 2-groups $\cA$ and $\cC$, choose
projective resolutions $\cP_{\cdot}\rightarrow\cA$ and
$\cQ_{\cdot}\rightarrow\cC$. By Lemma 2 and Lemma 3, there is a
projective resolution $\cP_{\cdot}\times\cQ_{\cdot}\rightarrow \cB$
fitting into an extension
$\cP_{\cdot}\xrightarrow[]{i_{\cdot}}\cP_{\cdot}\times\cQ_{\cdot}\xrightarrow[]{p_{\cdot}}\cQ_{\cdot}$
of projective complexes in (2-SGp)(\cite{1}). By Lemma 4, we obtain
a complexes of extension
$$T(\cP_{\cdot})\xrightarrow[]{T(i_{\cdot})}T(\cP_{\cdot}\times\cQ_{\cdot})\xrightarrow[]{T(p_{\cdot})}T(\cQ_{\cdot}).$$

Similar as Theorem 4.2 in \cite{11}, the long sequence
\begin{center}
\scalebox{0.9}[0.85]{\includegraphics{p28.eps}}
\end{center}
is 2-exact in each point.

\section*{Acknowledgements.}
We would like to give our special thanks to Prof. Zhang-Ju LIU,
Prof. Yun-He SHENG for very helpful comments. We also thank Prof. Ke
WU and Prof. Shi-Kun WANG for useful discussions.

\bibliographystyle{model1b-num-names}

\noindent Fang Huang, Shao-Han Chen, Wei Chen\\
Department of Mathematics\\
 South China University of
 Technology\\
 Guangzhou 510641, P. R. China

\noindent Zhu-Jun Zheng\\
Department of Mathematics\\
 South China University of
Technology\\
 Guangzhou 510641, P. R. China \\
 and\\
Institute of Mathematics\\
Henan University\\  Kaifeng 475001, P. R.
China\\
E-mail: zhengzj@scut.edu.cn

\end{document}